\documentclass[reqno,11pt]{amsart}

\usepackage{amssymb}
\usepackage[all,cmtip]{xy}
\usepackage{leftidx}
\usepackage{comment}
\usepackage{mathabx}
\usepackage{tikz-cd}
\usepackage[textsize=tiny]{todonotes}
\usepackage{MnSymbol} 
\usepackage[abbrev]{amsrefs}
\usepackage{amsthm}

\BibSpec{article}{%
  +{}{\PrintAuthors}  		{author}
  +{,}{ \textit}     		{title}
  +{,}{ }             		{journal}
  +{}{ \textbf}       		{volume}
  +{}{ \parenthesize} 		{date}
  +{}{, no. } 		{number}
  +{,}{ }      	      		{conference}
  +{,}{ }      	      		{book}
  +{,}{ }            		{pages}
  +{,}{ }            	 	{note}
  +{,}{ }            	 	{status}
  +{,}{  \texttt } {eprint}
  +{.}{}              {transition}
}

\BibSpec{book}{%
  +{}{\PrintAuthors}  {author}
  +{,}{ \textit}      {title}
  +{,}{ }             {publisher}
  +{,}{ }             {place}
  +{,}{ }             {date}
  +{.}{}              {transition}
}

\BibSpec{incollection}{%
    +{}  {\PrintAuthors}                {author}
    +{,} { \textit}                     {title}
    +{.} { }                            {part}
    +{:} { \textit}                     {subtitle}
    +{,} { \PrintContributions}         {contribution}
    +{,} { \PrintConference}            {conference}
    +{,} { }                            {booktitle}
    +{,} { pp.~}                        {pages}
    +{,}{ }       		{series}
    +{}{ \textbf}       		{volume}
    +{,} { }                            {publisher}
    +{,} { }                 {date}
    +{,} { }                            {status}
    +{,} { \PrintDOI}                   {doi}
    +{,} { available at \eprint}        {eprint}
    +{}  { \parenthesize}               {language}
    +{}  { \PrintTranslation}           {translation}
    +{;} { \PrintReprint}               {reprint}
    +{.} { }                            {note}
    +{.} {}                             {transition}
}

\newtheorem{Proposition}{Proposition}[section]
\newtheorem{Corollary}[Proposition]{Corollary}
\newtheorem{Definition}[Proposition]{Definition}
\newtheorem{Lemma}[Proposition]{Lemma}
\newtheorem{Theorem}[Proposition]{Theorem}
\newcounter{mt}

\newtheorem{MainTheorem}[mt]{Theorem}
\newtheorem{Remark}[Proposition]{Remark}
\newcounter{ct}

\DeclareMathOperator{\Val}{Val}

\DeclareMathOperator{\vol}{vol}

\DeclareMathOperator{\Gr}{Gr}

\DeclareMathOperator{\Span}{Span}
\DeclareMathOperator{\Dens}{Dens}
\DeclareMathOperator{\inv}{Inv}
\DeclareMathOperator{\ori}{or}
\DeclareMathOperator{\sgn}{sgn}

\DeclareMathOperator{\Ad}{Ad}
\DeclareMathOperator{\id}{id}
\DeclareMathOperator{\ad}{ad}
\DeclareMathOperator{\Sym}{Sym}

\DeclareMathOperator{\tr}{tr}

\DeclareMathOperator{\Inv}{Inv}
\DeclareMathOperator{\gr}{gr}
\DeclareMathOperator{\spt}{spt}
\newcommand{\CC}{\mathbb{C}}
\newcommand{\PP}{\mathbb{P}}
\newcommand{\C}{\mathbb{C}}
\newcommand{\p}{\mathbb{P}}
\newcommand{\R}{\mathbb{R}}
\newcommand{\RR}{\mathbb{R}}

\newcommand{\calK}{\mathcal{K}}

\newcommand{\calP}{\mathcal{P}}

\newcommand{\calV}{\mathcal{V}}
\newcommand{\calW}{\mathcal{W}}

\newcommand{\g}{\mathfrak g}

  \newcommand{\largewedge}{\mbox{\Large $\wedge$}}

\usepackage{xcolor}

\allowdisplaybreaks

\title{Invariant valuations on Lie groups}
	
\author[Andreas Bernig]{Andreas Bernig}
\author[Dmitry Faifman]{Dmitry Faifman}
\author[Jan Kotrbat{\'y}]{Jan Kotrbat{\'y}}

\address{Institut f\"ur Mathematik, Goethe-Universit\"at Frankfurt, Robert-Mayer-Str. 10, 60054 Frankfurt, Germany}
\email{bernig@math.uni-frankfurt.de}

\address{D\'epartement de Math\'ematiques et de Statistique, Universit\'e de
	Montr\'eal, CP 6128 succ Centre-Ville, Montr\'eal, QC H3C 3J7, Canada}
\email{dmitry.faifman@umontreal.ca}

\address{Charles University, Faculty of Mathematics and Physics, Mathematical Institute of Charles University, Sokolovsk\'a 49/83, 186 00 Prague, Czechia}
\email{kotrbaty@karlin.mff.cuni.cz}

\thanks{A.B. was supported by DFG grant BE 2484/10-1.}
\thanks{D.F. was supported by an NSERC Discovery grant and ISF grant No. 1750/20}
\thanks{J.K. was supported by CU grants PRIMUS/24/SCI/009 and UNCE/24/SCI/022.}

	\date{\today}

\begin{document}

	\begin{abstract} Convolution of valuations was introduced by the first named author and Fu for linear spaces, and later by Alesker and the first named author for compact Lie groups. In this paper we study the convolution of invariant valuations on Lie groups. First, we obtain an explicit formula for the convolution of left-invariant valuations on compact groups in terms of differential forms.
	Independently, we show that a connected Lie group  admits smooth bi-invariant valuations beyond the Euler characteristic and the Haar measure if and only if the group is the product of a compact group and a linear space. Finally, we use these two results to define the convolution of bi-invariant smooth valuations on an arbitrary unimodular Lie group, thus unifying both previously defined convolution operations.	
	\end{abstract}

		\maketitle
		

	\section{Introduction}

One of the most fundamental concepts in convex geometry is Minkowski addition of convex bodies.  The interplay between Minkowski sum and volume gives rise to the classical Brunn--Minkowski theory, pioneered in the first half of the twentieth century by Hilbert, Minkowski, Alexandrov, and others. In its core lies the Brunn--Minkowski inequality and, more generally, the notion of mixed volume and the Alexandrov--Fenchel inequality.

Much later on, a remarkable algebraic perspective on the classical geometric inequalities was discovered which in particular led to their far-reaching generalizations. Namely, it turns out that Minkowski sum restricted to a suitable class of convex bodies induces a natural algebra structure where the inequalities then arise as special cases of the (mixed) Hodge--Riemann relations. This was first carried out by  P.~McMullen for simple, strongly isomorphic polytopes \cite{mcmullen93}. Recently, the first and third named authors jointly with Wannerer accomplished the analogous program for smooth convex bodies in the realm of Alesker's algebraic theory of valuations \cite{bernig_kotrbaty_wannerer_HL_HR}.

Here a valuation is a function $\phi$ on the space of convex bodies satisfying
\begin{align*}
\phi(K\cup L)+\phi(K\cap L)=\phi(K)+\phi(L)
\end{align*}
for any convex bodies $K,L\subset\RR^n$ such that $K\cup L$ is convex. The graded Banach space $\Val(\RR^n)$ of all continuous, translation-invariant valuations contains a natural dense subspace $\Val^\infty(\RR^n)$ of smooth, translation-invariant valuations. By a recent result of Knoerr \cite{knoerr_finite}, the space $\Val^\infty(\RR^n)$ is spanned by valuations of the form $K\mapsto\vol(K+L)$, where $L\subset\RR^n$ is any convex body with smooth boundary and positive curvature. The Bernig--Fu convolution on $\Val^\infty(\RR^n)$ was defined in  \cite{bernig_fu06} as
\begin{align}
\label{eq:BFconvolution}
\vol(\bullet+K)*\vol(\bullet+L)=\vol(\bullet+K+L).
\end{align}

If $\RR^n$ is replaced by a general vector space $V$ without a canonical choice of the Lebesgue measure, \eqref{eq:BFconvolution} defines a product on $\Val^\infty(V)\otimes\Dens(V^*)$, where $\Dens(V^*)$ is the (one-dimensional) space of  Lebesgue measures on $V^*$. Equipped with the Bernig--Fu convolution, $\Val^\infty(V)\otimes\Dens(V^*)$ is a commutative, associative, graded algebra with unit $\vol\otimes\vol^*$. Denoting the vector addition on $V$ by $m$, the convolution can be equivalently defined as
\begin{align}
\label{eq:m*convolution}
\phi*\psi=m_*(\phi\boxtimes\psi),
\end{align}
where $\boxtimes$ denotes the exterior product of valuations and $m_*$ the push-forward under $m$, see Section \ref{ss:operations}. Besides the convolution, the space $\Val^\infty(V)$ carries the so-called Alesker product \cite{alesker04_product}. 
 Those two algebraic structures, which are in some precise sense dual to each other \cite{alesker_fourier, faifman_wannerer_fourier}, were key in much of the recent progress in integral geometry \cite{bernig_SU,bernig_G2_Spin7, bernig_fu11, bernig_hug,bernig_solanes_H2,kotrbaty_wannerer_O2, wannerer_area_measures, wannerer_unitary_module}.
 
The power of the algebraic theory of valuations relies on a deep result of Alesker \cite{alesker_irreducibility} which implies in particular that the elements of $\Val^\infty(V)$ can be represented by smooth differential forms. In these terms, the formula for the Bernig--Fu convolution becomes particularly simple, boiling down, essentially, to the wedge product.

Moreover, the language of differential forms provides a natural framework to extend the notion of smooth valuations to smooth manifolds \cite{alesker_val_man2,alesker_val_man4,alesker_intgeo,alesker_val_man3}. The filtered space of smooth valuations on a manifold $X$ is denoted by $\calV^\infty(X)$. Basic examples of smooth valuations are the Euler characteristic $\chi$ or any smooth measure on $X$. 

Let $\pi:\p_X\to X$ denote the cosphere bundle. By definition, each $\phi\in\calV^\infty(X)$ is represented by a pair $(\mu_\phi,\omega_\phi)\in \mathcal M^\infty(X)\times \Omega_{\ori}^{n-1}(\p_X)$,
 	where   $\mathcal M^\infty(X)$ are the smooth measures on $X$, and $\Omega_{\ori}(\p_X)$ denotes the space of complex-valued differential forms on $\p_X$ twisted by the pull-back by $\pi$ of the orientation line bundle of $X$. Alternatively, we can represent $\phi$ by a pair $(c_\phi,\tau_\phi)\in C^\infty(X)\times \Omega_{\ori}^n(\p_X)$.  We write $\phi=[[\mu_\phi,\omega_\phi]]=\{c_\phi,\tau_\phi\}$. Precise definitions and more details are given in Section \ref{ss:smooth}.
 	
While the Alesker product exists on any manifold $X$, in order to define the convolution, one needs, loosely speaking, a replacement for the Minkowski sum. In this regard, Alesker and the first named author \cite{alesker_bernig_convolution} proved that the formula \eqref{eq:m*convolution} defines an associative product on $\calV^\infty(X)$, called the convolution, if $X=G$ is a compact Lie group with multiplication $m:G\times G\to G$.
	
Due to the aforementioned importance of the translation-invariant case, it is a natural question whether the Alesker--Bernig convolution of valuations on compact Lie groups admits an analogous simple description in terms of differential forms. Our first main result answers this question affirmatively by establishing such a formula for invariant valuations.

Let $G$ be a compact Lie group with Lie algebra $\g$. Let $\calV^\infty(G)^{L_G}$ and $\calV^\infty(G)^{G\times G}$ denote the subspaces of  left- and bi-invariant valuations, respectively.  Similar notation will be used for spaces of differential forms. We have a natural identification
\begin{align*}
\Omega^p_{ \ori}(\p_G)^{L_G}\cong\bigoplus_k\Omega^k(\p_+(\g^*),\largewedge^{p-k} \g^*  \otimes \ori(\g)),\quad\tau\mapsto\sum_k\tau_k.
\end{align*}
Using the Haar probability measure on $G$, we further identify
\begin{align}
\label{eq:Haar_identification}
\largewedge^{p-k}\mathfrak g^*  \otimes \ori(\g)\cong \largewedge^{n-p+k} \mathfrak g \otimes \largewedge^n \mathfrak g^* { \otimes \ori(\g)} \cong \largewedge^{n-p+k}\mathfrak g,
\end{align}
and, given a form $\tau_k$ with values in $\largewedge^{p-k}\g^* \otimes \ori(\g)$,  write $\tilde\tau_k$ for the corresponding form with values in $\largewedge^{n-p+k}\g$. Let $e_1,\dots,e_n$ be a basis of $\g$, $e_1^*,\ldots,e_n^* \in \g^*$ the dual basis, and let $e_1^\sharp,\dots,e_n^\sharp$ denote the induced vector fields on $\p_+(\g^*)$. For $\tilde\tau,\tilde\zeta\in\Omega(\p_+(\mathfrak g^*),\largewedge^\bullet\mathfrak g)$ and $r\geq0$ we set
		\begin{align*}
			\hat S_r(\tilde \tau \otimes \tilde \zeta)= \sum_{|K|=r} \iota_{e_K^\sharp} \tilde\tau \wedge \iota_{e_K^*} \tilde\zeta,
		\end{align*}
where the sum runs over all sets $K=\{k_1,\dots,k_r\}\subset\{1,\dots,n\}$ and we define $\iota_{e_K^\sharp}=\iota_{e^\sharp_{k_1}}\circ\cdots\circ\iota_{e^\sharp_{i_r}}$ and similarly for $\iota_{e^*_K}$.

We define the convolution product $\tau * \zeta \in \Omega^{p+q-n}_{\ori}(\p_G)^{L_G}$ of differential forms $\tau \in \Omega^p_{\ori}(\p_G)^{G \times G}$ and $\zeta \in \Omega^q_{\ori}(\p_G)^{L_G} $ by 
	\begin{align}
	\label{eq:def_convolution_forms}
		\widetilde{(\tau * \zeta)}_j=\sum_{k+l \geq j} \epsilon^{p,q}_{k,l,j} \hat S_{k+l-j}(\tilde \tau_k \otimes \tilde \zeta_l),
	\end{align}
where 
\begin{align*}
	\epsilon_{k,l,j}^{p,q}=(-1)^{(n+q)(n+p+l+j)+k(l+j+1)}.
\end{align*}
 Observe that the lowest term in \eqref{eq:def_convolution_forms} is (up to a sign) the wedge product of $\tilde \tau_k$ and $\tilde \zeta_l$. We may thus think of this definition as a deformation of the usual wedge product.  We prove in Section \ref{s:conv_forms} that $\tau*\zeta\in\Omega_{\ori}(\p_G)^{G \times G}$ if $\zeta$ is bi-invariant, and that the convolution is associative.

\begin{MainTheorem}\label{mthm:compact}
Let $G$ be a compact Lie group. Let $\phi=\{c_\phi,\tau_\phi\} \in \mathcal V^\infty(G)^{G \times G}$ and $\psi=\{c_\psi,\tau_\psi\}=[[\mu_\psi,\omega_\psi]] \in \mathcal V^\infty(G)^{L_G}$ with $\omega_\psi$ left-invariant and $d\omega_\psi$ vertical. Denote $\mu(\psi)=\int_G \mu_\psi=\psi(G)\in \C$. Then 
\begin{align} 
\label{eq_main_formula_convolution}
	\phi * \psi=\{c_\phi \mu(\psi)+\pi_*(\tau_\phi * \omega_\psi), \tau_\phi * \tau_\psi\}.
\end{align}

\end{MainTheorem}	
We remark that vertical refers here to the canonical contact structure on $\p_G$ and that such $\omega_\psi$ always exists,  see Sections \ref{ss:notation} and \ref{ss:smooth}, respectively.
Formula \eqref{eq_main_formula_convolution} can in fact be extended to the case when also $\phi$ is only left-invariant. Indeed, it is easily seen that in such a case one has
\begin{align*}
	\phi * \psi=\int_G \mathbf{Ad}_g^*\phi dg \ * \psi,
\end{align*}
where $\mathbf{Ad}_g^*$ is the pull-back under the conjugation by $g\in G$, see Section \ref{s:compact}.

Equation \eqref{eq_main_formula_convolution} suggests that one might be able to define the convolution of bi-invariant valuations for non-compact Lie groups. We will see that this is indeed the case and that the convolution can be defined in such generality that it unifies the Alesker--Bernig convolution on compact Lie  groups with the Bernig--Fu convolution on linear spaces.

To this end, we first study which Lie groups have an interesting space of bi-invariant smooth valuations. On any Lie group, the Euler characteristic is a bi-invariant valuation. Furthermore, many Lie groups, such as semi-simple or nilpotent groups, are unimodular, i.e., possess a bi-invariant Haar measure. Our second main result classifies the connected Lie groups that admit smooth bi-invariant valuations beyond those two examples.
Henceforth, smooth valuations $\phi\notin\Span\{\chi, \vol\}$ will be called \emph{non-obvious}. 
 \begin{MainTheorem} \label{mainthm_existence_biinvariant}
 	A  connected Lie group $G$ of dimension at least 2 admits a  smooth non-obvious bi-invariant valuation if and only if $G$ admits a bi-invariant riemannian metric. In this case, the space $\mathcal V^\infty(G)^{G \times G}$ is infinite-dimensional except for $G=\mathrm{SO}(3)$ and $G=S^3$.  	
 \end{MainTheorem}	

 In particular, such groups are unimodular. For an equivalent explicit description, recall that by a result of Milnor \cite{milnor76}, a connected Lie group $G$ admits a bi-invariant riemannian metric if and only if $G$ is the cartesian product of a compact group and a linear space. 

In the two exceptional cases $G=S^3, \,\mathrm{SO}(3)$, the space of bi-invariant valuations coincides with the space of intrinsic volumes for the bi-invariant round metric, see for example  \cite{bernig_faifman_solanes_part2}. We compute the full convolution table for $S^3$ in Section \ref{s:S3}.

We now turn to defining the convolution of smooth bi-invariant valuations on general unimodular groups. As there is no canonical choice of a Haar measure to play the role of unit element, it is natural to introduce a twist by the one-dimensional space $\Dens(\g^*)$ of dual densities on the Lie algebra $\g$ or, equivalently, the dual to the space of Haar measures on $G$.
	
It is not hard to see that the Haar measures on $G$ naturally form a direct summand in $\mathcal V^\infty(G)^{L_G}$, see Proposition \ref{prop:haar_direct}. We denote the Haar measure component by $\mu(\bullet)$ and retain the same notation also for spaces of valuations twisted by dual densities, so that for $\phi\in\calV^\infty(G)^{L_G}\otimes \Dens(\mathfrak g^*)$, we have $\mu(\phi)\in \mathbb C$, as in Theorem \ref{mthm:compact}.  Further, replacing \eqref{eq:Haar_identification} by
	\begin{displaymath}
		\largewedge^{p-k}\mathfrak g^* {\otimes \ori(\g)} \otimes \Dens(\mathfrak g^*) \cong \largewedge^{n-p+k}\mathfrak g,
	\end{displaymath} 
	$\tilde\tau_k$ will denote the form with values in $\largewedge^{n-p+k}\g$ corresponding to a form $\tau_k\in\Omega^k_{\ori}(\p_+(\g^*),\largewedge^{p-k}\g^*)\otimes \Dens(\mathfrak g^*)$. Formula \eqref{eq:def_convolution_forms} then defines an associative convolution product on $\Omega_{\ori}(\p_G)^{G \times G}\otimes \Dens(\mathfrak g^*)$, and turns the space $\Omega_{\ori}(\p_G)^{L_G}$ into a left module over the  algebra $\Omega_{\ori}(\p_G)^{G \times G}\otimes \Dens(\mathfrak g^*)$.

\begin{Definition}\label{mdef:convolution}
	Let $G$ be a unimodular Lie group. The convolution of valuations $\phi\in \mathcal V^\infty(G)^{G \times G}\otimes \Dens(\mathfrak g^*)$ and $\psi\in \mathcal V^\infty(G)^{L_G}$ is defined as
	\begin{equation} 
		\label{eq_def_formula_convolution}
		\phi * \psi=\{c_\phi \mu(\psi)+\pi_*(\tau_\phi * \omega_\psi), \tau_\phi * \tau_\psi\},
	\end{equation}
where $\phi=\{c_\phi,\tau_\phi\}$ and $\psi=\{c_\psi,\tau_\psi\}=[[\mu_\psi,\omega_\psi]]$ with $\omega_\psi$ left-invariant and $d\omega_\psi$ vertical.
	\end{Definition}	
The validity of Definition \ref{mdef:convolution} is not evident; namely, one has to verify that a form $\omega_\psi$ with the required properties exists, the definition does not depend on its choice, and that the right-hand side of \eqref{eq_def_formula_convolution} indeed corresponds to a valuation. Doing this, together with proving the main properties of the so-defined convolution product, is our third main result.

\begin{MainTheorem}
	\label{mthm:unimodular}
	Let $G$  be a unimodular Lie group. Then the convolution of Definition \ref{mdef:convolution} is well defined and continuous, gives rise to an associative product on $\mathcal V^\infty(G)^{G \times G}\otimes \Dens(\mathfrak g^*)$, and turns the space $\mathcal V^\infty(G)^{L_G}$ into a module over the convolution algebra $\mathcal V^\infty(G)^{G \times G}\otimes \Dens(\mathfrak g^*)$. Moreover, it satisfies the following properties:
	\begin{enumerate}
		\item $\vol\otimes\vol^*\in \mathcal V^\infty(G)^{G \times G}\otimes \Dens(\mathfrak g^*)$ is the unit element.
		\item The convolution is compatible with the canonical filtration $\calV^\infty(G)=\calW_0\supset\calW_1\supset\cdots\supset\calW_n=\mathcal M^\infty(G)$ in the sense that
		\begin{align*}
			\left(\mathcal W_k^{G \times G}	\otimes \Dens(\mathfrak g^*)\right) * \mathcal W_l^{L_G}	 \subset \mathcal W_{k+l-n}^{L_G}.
		\end{align*} 
		\item For any $\phi\in\mathcal V^\infty(G)^{G \times G} \otimes \Dens(\mathfrak g^*)$ and $\psi\in\mathcal V^\infty(G)^{L_G}$ it holds that $(\chi \otimes \vol^*)* \psi=(\mu(\psi)\vol^*) \chi$ and $\phi\ast\chi= \mu(\phi)\chi$. In particular, $\Span\{\chi\otimes\vol^*\} \subset \mathcal V^\infty(G)^{G \times G} \otimes \Dens(\mathfrak g^*)$ is a two-sided nilpotent ideal.
		\item \emph{Alesker--Poincar\'e duality.} If $G$ is connected, then the $\Dens(\mathfrak g^*)$-valued pairing on $\mathcal V^\infty(G)^{G\times G}\otimes \Dens(\mathfrak g^*)$, $(\phi,\psi)\mapsto(\phi*\psi)(e)$ is perfect. Moreover, one has $(\phi*\psi)(e) =\mu(\phi\cdot \inv^*\psi)$ where the dot denotes the Alesker product and $\inv:G \to G$ the inverse map. 
	\end{enumerate}
\end{MainTheorem}

The proof of Theorem \ref{mthm:unimodular} we give in Section \ref{sec:convolution_unimodular} takes the following indirect approach. We first define an a-priori distinct convolution operation on groups of the form $G=K\times V$, with $K$ compact and $V$ linear, exploiting the fact that for $\Gamma\subset V$ a lattice, the quotient 
$G\to K\times V/\Gamma$ induces an isomorphism on the corresponding spaces of bi-invariant valuations. We then proceed to show that it coincides with the convolution of Definition \ref{mdef:convolution}. Finally, the classification established in Theorem \ref{mainthm_existence_biinvariant} allows us to complete the proof of Theorem \ref{mthm:unimodular} essentially by reduction to the compact case.

We note without a proof that it can be shown directly that  \eqref{eq_def_formula_convolution} gives rise to a well-defined convolution on $\mathcal V^\infty(G)^{G \times G}	\otimes \Dens(\mathfrak g^*)$ for any unimodular group $G$, without relying on Theorem \ref{mainthm_existence_biinvariant}. 
The advantage of such an a-priori approach is that, under some further conditions such as compatibility of wavefront sets, it applies also to non-smooth bi-invariant valuations, which exist for a larger class of Lie groups. For instance, on semi-simple Lie groups one has a bi-invariant pseudo-riemannian structure, with corresponding intrinsic volumes which are generalized valuations \cite{bernig_faifman_solanes_pseudo}. In the smooth case, however, this task is significantly simplified by making use of Theorem \ref{mainthm_existence_biinvariant}.

Finally, we establish the compatibility of Definition \ref{mdef:convolution} with the previously defined convolutions. Besides proving that it extends the Alesker--Bernig convolution on compact Lie groups and the Bernig--Fu convolution on flat spaces, we also show that the infinitesimal version of our convolution agrees with the Bernig--Fu convolution of $\Ad_G$-invariant valuations on the corresponding Lie algebra and that a certain natural averaging map is a homomorphism of convolution algebras.

Observe that a Lie group $G$ that is isomorphic to the cartesian product of a compact group and a linear space,  has a unique maximal compact subgroup $K\subset G$. We then define $A:\mathcal V^\infty_c(G)^{K\times K}\to \mathcal V^\infty(G)^{G\times G}\otimes\Dens(\mathfrak g^*)$ by  $A(\phi)=\int_G L_g^*\phi d\sigma(g) \otimes \sigma^*$ for an arbitrary Haar measure $\sigma\in\Dens(\mathfrak g)$. Here $\calV^\infty_c$ denotes the subspace of compactly supported valuations, $L_g^*$ the pull-back under the left translation by $g\in G$, and $\sigma^*$ the corresponding dual density.

 \begin{Proposition}
 \label{mthm:compatibility}
Let $G$ be a unimodular Lie group. Consider the convolution on $\mathcal V^\infty(G)^{G \times G} \otimes \Dens(\mathfrak g^*)$ of Definition \ref{mdef:convolution}.
 		\begin{enumerate}
 		\item  If $G$ is compact, then the map $\mathcal V^\infty(G)^{G \times G}\to \mathcal V^\infty(G)^{G \times G} \otimes \Dens(\mathfrak g^*)$ given by $\phi \mapsto \phi \otimes \vol(G) \vol^*$ is an isomorphism of algebras.    
 		\item \label{item_flat_case} If $G=V$ is a vector space, then the convolution on
 		\begin{align*}
 		\calV^\infty(V)^{V\times V}\otimes\Dens(V^*)=\Val^\infty(V)\otimes\Dens(V^*)
 		\end{align*}
 		agrees with the Bernig--Fu convolution. 
  \item There exists an isomorphism of graded algebras  
    \begin{align*}
    	\gr  \mathcal V^\infty(G)^{G \times G} \otimes \Dens(\mathfrak g^*) \cong \Val^\infty(\mathfrak g)^{\Ad G} \otimes \Dens(\mathfrak g^*).
    \end{align*}
    \item   Let $G$  be isomorphic to the product of a compact Lie group and a linear space, and let $K\subset G$ be the maximal compact subgroup. Then for all $\phi,\psi\in\mathcal V^\infty_c(G)^{K\times K}$ and $\eta\in\mathcal V^\infty(G)^{G\times G}\otimes \Dens(\mathfrak g^*)$ it holds that $$A(\phi\ast\psi)=A\phi\ast\psi=\phi * A\psi=A\phi\ast A\psi$$ and $\phi\ast\eta=A\phi\ast \eta$.
\end{enumerate} 
\end{Proposition}

\subsection*{Acknowledgements.} We are grateful to Semyon Alesker for numerous illuminating conversations.  D.F. is indebted to Leonid Rybnikov for some helpful explanations on Lie algebras. A large part of this work was carried out during  D.F.'s term at Tel Aviv University,  A.B.'s research stay at IHES, and our visit at the R\'enyi institute and we wish to thank those institutions for their support and productive atmosphere.

\section{Preliminaries}

\subsection{Notation}

\label{ss:notation}
 
For a real vector space $V$, let $\p(V)$ be the projective space consisting of lines in $V$, and $\p_+(V)$ its double cover of oriented lines. Thus $\p_+(V)$ is the set of equivalence classes of vectors $v \in V \setminus \{0\}$ with respect to the equivalence relation $v \cong \lambda v, \lambda>0$. The equivalence class of $v$ will be denoted by $[v]$. For a smooth manifold $X$, we let $\p_X$ denote the cosphere bundle, consisting of all pairs $(x,[\xi])$ with $x \in X, \xi \in \p_+(T_x^*X)$. It is naturally a contact manifold, with the contact distribution given at the point $(x,[\xi])$ by $d\pi|_{(x,[\xi])}^{-1}(\ker \xi)$, where $\pi:\mathbb P_X\to X$ is the projection.

We will use the convention for the push-forward of differential forms as in \cite{alesker_bernig, alesker_bernig_convolution}. More precisely, if $\pi:M \to A$ is a smooth fiber bundle with compact fibers and $M,A$ oriented manifolds, then 
\begin{displaymath}
	\int_A \alpha \wedge \pi_*\omega=\int_M \pi^*\alpha \wedge \omega
\end{displaymath}
for $\alpha\in\Omega(A)$ compactly supported and $\omega\in\Omega(M)$. Note that the degree of $\pi_*\omega$ equals the degree of $\omega$ minus the degree of the fiber. With this convention, we have $d \circ \pi=\pi \circ d$ and the formula
\begin{displaymath}
	\pi_*(\pi^*\alpha \wedge \omega)=\alpha \wedge \pi_*\omega.
\end{displaymath}

To compute $\pi_*$ explicitly, assume $\pi_*\omega\in\Omega^k(M)$ and take a point $a \in A$ and tangent vectors $v_1,\ldots,v_k \in T_aA$. For a point in the fiber of $a$, take lifts $\tilde v_1,\ldots,\tilde v_k$. We then have 
\begin{displaymath}
	\pi_*\mu(v_1,\ldots,v_k)=\int_{\pi^{-1}(a)} \mu(\tilde v_1,\ldots,\tilde v_k,\bullet),   
\end{displaymath}
where the fiber is oriented in the usual way such that, locally, the orientation of $M$ is the product orientation corresponding to $A\times\pi^{-1}(a)$.

Note that all scalar-valued (differential) forms considered throughout the paper take values in $\CC$. Similarly, all tensor products are over $\CC$.

\subsection{Lie groups}

Let $G$ be an $n$-dimensional real Lie group  and $\mathfrak g$ its Lie algebra. Take $g,h\in G$. We use the  notation as in \cite{duistermaat_kolk} $L_hg=hg$, $ R_hg=gh$, and 
$\mathbf{Ad}_h(g)=hgh^{-1}$. We set \[L_h^*:=(dL_{h^{-1}})^*:T_g^*G \to T_{hg}^*G,\quad R_h^*:=(dR_{h^{-1}})^*:T_g^*G \to T_{gh}^*G,\] and \[\Ad_h^*:=(d\mathbf{Ad}_{h^{-1}})^*:T_g^*G \to T_{hgh^{-1}}^*G.\] Similarly, we denote the adjoint action $\Ad_h=d\mathbf{Ad}_h:\mathfrak g \to \mathfrak g$. We then have $\Ad_h^*=(\Ad_{h^{-1}})^*:\mathfrak g^* \to \mathfrak g^*$. For $x\in\g$ we set $\ad_x=\frac d{dt}\big|_{t=0}\Ad_{\exp(tx)}:\g\to\g$ and $\ad_x^*:=\frac d{dt}\big|_{t=0}\Ad_{\exp(tx)}^*=(-\ad_x)^*$, and let $x^\sharp$ be the vector field on  $\p_+(\mathfrak g^*)$ defined by 
	\begin{displaymath}
		x^\sharp|_{[\xi]}=\left.\frac{d}{dt}\right|_{t=0} [\Ad_{\exp(tx)}^* \xi]=[\ad^*_x(\xi)].
	\end{displaymath}

By $\ori(\g)$ we denote the line of orientations on $\g$. The connected components of $\ori(\g)\setminus\{0\}$ correspond to orientations of $\mathfrak g$, which can be identified with the left-invariant orientations on $G$ using left translation. Thus $\ori(\g)\simeq \ori(\g^*)\simeq \ori(\g)^*$ is a real 1-dimensional module over $G$, such that $g\in G$ acts by the scalar $\sgn \det (\Ad_g: \g\to \g)$.  There is a natural identification $\Dens(\mathfrak g)=\largewedge^n \g^*\otimes \mathrm{or}(\g)$. The group $G$ is called \emph{unimodular} if it has a bi-invariant Haar measure, or equivalently if $|\det \Ad_g|=1$ for all $g\in G$. It then holds that
\begin{align}
	\label{eq:trace}
	\mathrm{tr}(\ad_X)=0,\quad X\in \mathfrak g,
\end{align}
and the inverse map $\inv:G\to G$ preserves any given Haar measure. In particular, compact groups are unimodular.

Let $*:\largewedge^k \g\otimes \largewedge ^n\g^* \to \largewedge^{n-k}\mathfrak g^*$ be the Hodge star isomorphism given by
\begin{displaymath}
	\langle *(X\otimes\det),Y\rangle=\det(X \wedge Y), \quad Y \in \largewedge^{n-k} \mathfrak g.
\end{displaymath}
Using that $\largewedge ^n\g^*\otimes \largewedge ^n\g\cong \C$, we thus have the isomorphism
\begin{align}
\label{eq:Hodge-1}
*^{-1}:\largewedge^k\g^*\otimes \largewedge^n\g\to \largewedge^{n-k}\g,
\end{align}
which will be frequently used in the paper. Observe that if $\tau \in \largewedge^k \mathfrak g^* \otimes \largewedge^n \mathfrak g$ and $X_1,\ldots,X_k \in \mathfrak g$, then 
\begin{equation} \label{eq_star_explicit}
	\tau(X_1,\ldots,X_k)=*^{-1}\tau \wedge X_1 \wedge \ldots \wedge X_k.
\end{equation}
By naturality of the Hodge star we have
\begin{align}
\label{eq:AdHodge}
	\Ad_h^* \circ\, *=* \circ \Ad_h.
\end{align}
Differentiating, we infer that
\begin{align}
\label{eq:adHodge}
 \ad_x^* \circ *=* \circ \ad_x,\quad x\in\g.
\end{align}
Finally, the Koszul boundary operator $\partial:\wedge^k\mathfrak g\to \wedge^{k-1}\mathfrak g$ is defined by  
\begin{displaymath}
	\partial(x_1 \wedge \ldots \wedge x_k)=\sum_{1 \leq i< j \leq k} (-1)^{i+j+1} [x_i,x_j] \wedge x_1 \wedge \cdots \wedge \widecheck x_i \wedge \cdots \wedge \widecheck x_j \wedge \cdots \wedge x_k,
\end{displaymath}
 where $\widecheck x_i$ means that $x_i$ is omitted. It satisfies $\partial ^2=0$.

\subsection{Smooth valuations}

\label{ss:smooth}

We refer to \cite{alesker_val_man2, alesker_intgeo, alesker_barcelona, alesker_val_man3} for the theory of smooth valuations.  

Let $X$ be an $n$-dimensional smooth manifold and  $\p_X=\PP_+(T^*X)$ its cosphere bundle. Denote by $\pi:\mathbb P_X\to X$ the projection. Let $\calP(X)$ denote the class of compact differentiable polyhedra in $X$. Each $P\in\calP(X)$ admits a normal cycle $N(P)$ which is a naturally oriented Legendrian $(n-1)$-dimensional Lipschitz submanifold of $\p_X$, see \cite{alesker_val_man3} for details.

We denote by $\ori(X)$ the orientation line bundle on $X$. The line bundles underlying smooth measures and top differential forms on $X$ are related by $\Dens(TX)=\largewedge^nT^*X\otimes\ori(X)$. By $\Omega_{\ori}(X)$, resp. $\Omega_{\ori}(\mathbb P_X)$,  we denote the space of complex-valued differential forms on X, resp. on $\p_X$, twisted by $\ori(X)$, resp. by $\pi^*\ori(X)$. Note that $\Omega^n_{\ori}(X)=\mathcal M^\infty(X)$.

A functional $\phi:\calP(X)\to\CC$ is called a smooth valuation if there exist $\mu\in\Omega^n_{\ori}(X)$ and $\omega\in\Omega^{n-1}_{\ori}(\p_X)$ such that
\begin{align*}
\phi(P)=\int_P\mu+\int_{N(P)}\omega,\quad P\in\calP(X).
\end{align*}
In this case we denote $\phi=[[\mu,\omega]]$.

Clearly, any measure on $X$ is a smooth valuation. Another example is the Euler characteristic $\chi$ for which the corresponding forms were constructed by Chern \cite{chern45}.  If $X$ is endowed with a riemannian metric, then the intrinsic volumes $\mu_k,k=0,\ldots,n$, are smooth valuations on $X$. The Fr\'echet space of smooth valuations on $X$ is denoted by $\calV^\infty(X)$. It admits a canonical filtration by closed subspaces
\begin{displaymath}
	\calV^\infty(X)=\calW_0\supset\calW_1\supset\cdots\supset\calW_n= \Omega^n_{\ori}(X),
\end{displaymath}
such that the associated graded space $\gr(\calV^\infty(X))=\bigoplus_{k=0}^n\calW_k/\calW_{k+1}$ is canonically isomorphic to the space of smooth sections of the vector bundle
\begin{displaymath}
	\Val^\infty(TX)\to X
\end{displaymath}
whose fiber over $x\in X$ is the space $\Val^\infty(T_xX)$ of smooth and translation-invariant valuations, cf. Section \ref{ss:Val} below.

By definition, the map $\Omega^n_{\ori}(X) \times \Omega_{\ori}^{n-1}(\p_X) \to \mathcal V^\infty(X), (\mu,\omega) \mapsto [[\mu,\omega]]$ is surjective. Let us describe the kernel, following \cite{bernig_broecker07}. A differential form on $\p_X$ is called vertical if its restriction to the contact distribution vanishes. Rumin has constructed a differential operator $D:\Omega^{n-1}(\p_X) \to \Omega^n(\p_X)$ of order $2$ such that $D\omega=d(\omega+\xi)$, where $\xi$ is the unique vertical $(n-1)$-form such that $d(\omega+\xi)$ is vertical \cite{rumin94}. We then have $[[\mu,\omega]]=0$ if and only if $D\omega+\pi^*\mu=0$ and $\pi_*\omega=0$, where $\pi:\p_X\to X$ is the bundle projection. Any smooth valuation $[[\mu,\omega]]$ can thus be equivalently described by the pair $(\pi_*\omega, D\omega+\pi^*\mu) \in C^\infty(X)\times \Omega^n_{\ori}(\p_X)$. Note that $(f,\tau) \in  C^\infty(X)\times \Omega^n_{\ori}(\p_X)$ corresponds to a (unique) smooth valuation if and only if $\tau$ is vertical and exact, and $\pi_*\tau=df$. In this case we write $\{f,\tau\}$ for the valuation defined by this pair. With this notation, we have
\begin{align*}
	[[\mu,\omega]]=\{\pi_*\omega,D\omega+\pi^*\mu\}.
\end{align*}

\subsection{Operations on smooth valuations}
\label{ss:operations}

The space $\calV^\infty(X)$ carries a natural bilinear structure, the Alesker product, which turns it into a commutative associative filtered algebra with unit $\chi$, see \cite{alesker04_product, alesker_bernig, alesker_val_man3} for its construction. A fundamental property of the Alesker product is the following version of Poincar\'e duality. Let $\calV^\infty_c(X)\subset\calV^\infty(X)$ denote the ideal of compactly supported valuations, i.e., $\phi\in\calV_c^\infty(X)$ if it vanishes outside of a compact set $P_0$. In this case $\int_X\phi:=\phi(P)$ does not depend on $P\in\mathcal P(X)$ if $P\supset P_0$. Then the bilinear map 
\begin{displaymath}
\calV^\infty(X)\times \calV_c^\infty(X)\to\CC,\quad (\phi,\psi)\mapsto\int_X\phi\cdot\psi
\end{displaymath}
is a perfect pairing. Consequently, $\calV^\infty(X)$ embeds densely in the space 
\begin{equation} \label{eq_def_gen_vals}
	\calV^{-\infty}(X):=\left(\calV^\infty_c(X)\right)^*
\end{equation}
of generalized valuations.

Given two smooth manifolds $X,Y$, there exists a natural bilinear map $\boxtimes: \mathcal V^\infty(X) \times \mathcal V^\infty(Y) \to \mathcal V^{-\infty}(X \times Y)$, called exterior product. Note that the image is not a smooth valuation, but merely a generalized valuation in the sense of \eqref{eq_def_gen_vals}. For a smooth map $f:X \to Y$, there are further partially defined maps $f_* : \mathcal V^{-\infty}(X) \dashedrightarrow \mathcal V^{-\infty}(Y)$ and $f^*: \mathcal V^{-\infty}(Y) \dashedrightarrow \mathcal V^{-\infty}(X)$, called push-forward and pull-back, respectively, under $f$.  

The product of smooth valuations on a manifold $X$ is the pull-back of the exterior product under the diagonal map $X \hookrightarrow X \times X, x \mapsto (x,x)$, and it is again a smooth valuation on $X$. If a Lie group $G$ acts transitively on a manifold $X$, then the push-forward of the exterior product of a smooth valuation on $G$ and a smooth valuation on $X$ under the map $G \times X \to X$ is called the convolution, and it is again a smooth valuation on $X$. 

The Euler--Verdier involution $\sigma:\mathcal V^\infty(X) \to \mathcal V^\infty(X)$ is defined by
	\begin{displaymath}
		\sigma [[\mu,\omega]]:=(-1)^n [[\mu,a^*\omega]],
	\end{displaymath}
where $n=\dim X$ and $a:\p_X \to \p_X$ is the antipodal map. It is compatible with the filtration and the Alesker product and it extends to an involution on $\mathcal V^{-\infty}(X)$. The intrinsic volumes satisfy $\sigma \mu_k=(-1)^k\mu_k$.

\subsection{Translation-invariant valuations}
\label{ss:Val}

Let $X=V$ be a real vector space and let $\calK(V)$ denote the set of convex bodies, i.e., compact convex subsets of $V$. According to a deep result of Alesker, the subspace $\Val^\infty(V)\subset\calV^\infty(V)$ of translation-invariant valuations coincides with the dense subspace of $\mathrm{GL}(V)$-smooth vectors in the Banach space $\Val(V)$ of translation-invariant \emph{continuous} valuations on $\calK(V)$. In fact, by a recent result of Knoerr \cite{knoerr_finite},
\begin{displaymath}
  \Val^\infty(V)=\Span\{K\mapsto\vol(K+A)\mid A\in\calK^\infty_+(V)\},
\end{displaymath}
where $\vol\in\Dens(V)$ is a Lebesgue measure and $\calK^\infty_+(V)\subset\calK(V)$ consists of convex bodies with smooth boundary and positive curvature.
Observe that one has the so-called McMullen grading
\begin{displaymath}
\Val^\infty(V)=\bigoplus_{k=0}^n\Val^\infty_k(V),
\end{displaymath}
where $\Val_k^\infty (V)\subset\Val^\infty$ is the subspace of $k$-homogeneous valuations, i.e., $\phi\in\Val_k^\infty$ if $\phi(\lambda K)=\lambda^k\phi(K)$ for any $\lambda>0$ and $K\in\calK(V)$.

The space  $\Val^\infty(V)$ carries besides the Alesker product another algebraic structure. More precisely, the Bernig--Fu convolution is the commutative associative bilinear product on $\Val^\infty(V) \otimes\Dens(V^*)$ given by
$$
\left[\vol_n(\cdot+A_1)\otimes\mu_1^*\right]*\left[\vol_n(\cdot+A_2)\otimes\mu_2^*\right]=\mu_1^*(\vol_n)\vol_n(\cdot+A_1+A_2)\otimes\mu_2^*.
$$
It satisfies
\begin{displaymath}
  [\Val^\infty_k(V)\otimes\Dens(V^*)] * [\Val^\infty_l(V)\otimes\Dens(V^*)] \subset\Val^\infty_{k+l-n}(V)\otimes\Dens(V^*).
\end{displaymath}

\section{Invariant differential forms}

\subsection{Left-invariant and bi-invariant forms}

Let $G$ be a unimodular Lie group of dimension $n$ with Lie algebra $\mathfrak g$. In this section we will study invariant differential forms on the cosphere bundle $\p_G=\p_+(T^*G)$ over $G$. 

We will be mostly utilizing the two actions of $G$ on itself given by left translation and conjugation, which induce corresponding actions on $\p_G$ in the natural way:
\begin{align*}
L_h((g,[\xi]))=(hg,[L_h^*\xi])\quad\text{and}\quad \mathrm{Ad}_h((g,[\xi]))=(\mathbf{Ad}_h(g),[\Ad_h^*\xi]).
\end{align*} 
A differential form on $\p_G$ is called left-, resp. $\Ad$-invariant, if it is invariant under left translation, resp. $\mathrm{Ad}_G$. A form is bi-invariant if it is invariant under both left and right translation, or equivalently if it is both left- and $\Ad_{ G}$-invariant.
Using the diffeomorphism $G\times \mathbb P_+(\mathfrak g^*)\to \mathbb P_G$, $(g,[\xi])\mapsto(g, [L_g^*\xi])$, the corresponding induced actions on $G\times \mathbb P_+(\mathfrak g^*)$ are $L_h \times \mathrm{id}$ and $\mathbf{Ad}_h\times\Ad_h^*$.

Observe that we then have the following identifications:
\begin{align*}
	\Omega^p(\mathbb P_G)^{L_G} & \cong \bigoplus_k\Omega^{p-k,k}(G\times \mathbb P_+(\mathfrak g^*))^{L_G} \\
	& \cong \bigoplus_k\Omega^k(\p_+(\mathfrak g^*)) \otimes \Omega^{p-k}(G)^{L_G}\\
	& \cong \bigoplus_k \Omega^{k}(\mathbb P_+(\mathfrak g^*), \largewedge^{p-k} \mathfrak g^*).
\end{align*}
For $\tau\in \Omega^p(\mathbb P_G)^{L_G}$, we decompose accordingly $\tau=\sum_k\tau_k$, where
	\begin{equation}
		\tau_k(u_1,\dots, u_k)(x_1,\dots, x_{p-k}) =\tau(u_1,\dots, u_k,x_1,\dots, x_{p-k}). \label{eq_homogeneous_decomposition}
	\end{equation}
Here $u_1,\dots, u_{k}$ are vector fields on $\mathbb P_+(\mathfrak g^*)$ and $x_1,\dots, x_{p-k}$ left-invariant vector fields on $G$, which we identify with vector fields on $\p_G \cong G \times \p_+(\mathfrak g^*)$.

Considered as an action on  $\Omega(\p_+(\mathfrak g^*)) \otimes \Omega(G)^{L_G}$, the adjoint action $(\mathbf{Ad}_h\times\Ad_h^*)^*$ takes the pull-back of the form under the diffeomorphism $\Ad_h^*$ on $\p_+(\mathfrak g^*)$ and applies $(\mathbf{Ad}_h)^*$ to the value. Differentiating in the second factor then induces the action $(\Ad_h^*)^* \otimes \Ad_{h^{-1}}^*$ on $\Omega(\mathbb P_+(\mathfrak g^*), \largewedge^\bullet \mathfrak g^*)$. A form $\tau\in\Omega(\mathbb P_G)^{L_G}$ is thus bi-invariant if and only if the corresponding form $\tau\in\Omega(\mathbb P_+(\mathfrak g^*), \largewedge^\bullet \mathfrak g^*)$ is $\Ad$-invariant, that is if for any $h\in G$,
\begin{align}
\label{eq:right-inv}
(\Ad_h^*)^* \otimes \Ad_{h^{-1}}^*\tau=\tau.
\end{align}

Finally, take \[\tau\in\Omega^{k}(\mathbb P_+(\mathfrak g^*), \largewedge^{p-k} \mathfrak g^* \otimes\ori(\mathfrak g))\otimes\Dens(\mathfrak g^*)=\Omega^{k}(\mathbb P_+(\mathfrak g^*), \largewedge^{p-k} \mathfrak g^*)\otimes\largewedge^n\mathfrak g.\] The following notation will be used throughout the paper, see \eqref{eq:Hodge-1}:
\begin{equation} \label{eq_def_hodge_star}
	\tilde\tau=*^{-1}\tau\in \Omega^{k}(\mathbb P_+(\mathfrak g^*), \largewedge^{n-p+k} \mathfrak g).
\end{equation}

\subsection{The differential on invariant forms}

In the following we will consider the decomposition of left-invariant forms defined by \eqref{eq_homogeneous_decomposition}. As usual we set $\tau_k=0$ for $k<0$.

Observe that under the identification of $\Omega^k(\p_+(\g^*),\largewedge^{p-k} \g^*)$ with a subspace of $\Omega^p(\p_G)^{L_G}$, the de Rham differential on the former space differs from the restriction of the de Rham differential on the latter. To avoid any ambiguity, we will adhere to the following convention in the rest of the paper: When we write $\tau\in\Omega^p(\p_G)^{L_G}$, then $d\tau$ refers to the differential of the (complex-valued) $p$-form on $\p_G$, whereas for $\tau\in\Omega^k(\p_+(\g^*),V)$, where $V$ is a vector space, $d\tau$ denotes the differential of the $V$-valued $k$-form on $\p_+(\g^*)$.

\begin{Lemma}\label{lem:closed_compatible}
	Let $0 \leq p \leq 2n-1$, $\tau \in \Omega^p(\p_G)^{L_G}$ and $\tau_k \in \Omega^k(\p_+(\mathfrak g),\largewedge^{p-k}\mathfrak g^*)$ as in \eqref{eq_homogeneous_decomposition}. For $0\leq k\leq p$ we have
\begin{equation} \label{eq_d_and_partial}
	(d\tau)_k=d\tau_{k-1}+(-1)^{k+1} \partial^* \tau_k. 
\end{equation}	
\end{Lemma}

\proof
 Let $x_1,\dots, x_{p-k+1}$ be left-invariant vector fields on $G$ and $u_1,\dots, u_k$ commuting vector fields on some open subset $U\subset\mathbb P_+(\mathfrak g^*)$, which we identify with left-invariant vector fields near $\{e\} \times U \subset G\times\mathbb P_+(\mathfrak g^*)\cong\p_G$.  Since $\tau$ is left-invariant, it holds for all $i$ that
\begin{displaymath}
	x_i \tau (u_1,\ldots,u_k,x_1,\dots, \widecheck x_i, \dots, x_{p-k+1})=0,
\end{displaymath}
where $\widecheck x_i$ means that $x_i$ is omitted. Note also that $[u_i, x_j]=0$ for all $i_,j$. Therefore, 
\begin{align*}
	&d\tau(u_1,\ldots,u_k,x_1,\ldots,x_{p-k+1}) \\
	&=\sum_i (-1)^{i+1} u_i \tau ( u_1,\dots,\widecheck u_i,\dots, u_k,x_1,\dots, x_{p-k+1})\\
	&\quad+\sum_{i<j} (-1)^{i+j+k} \tau (u_1,\dots, u_k, [x_i,x_j],x_1\dots, \widecheck x_i, \dots, \widecheck x_j, \dots, x_{p-k+1} )\\
	&=d\tau_{k-1}(u_1,\ldots,u_k)(x_1,\ldots,x_{p-k+1})\\
	&\quad+(-1)^{k+1} \partial^* \tau_{k+1}(u_1,\ldots,u_k)(x_1,\ldots,x_{p-k+1}).
\end{align*} 
\endproof

\begin{Lemma}\label{lem:partial_leibnitz}
	For $X\in\wedge^{k+1}\mathfrak g$, $Y\in\wedge^{n-k}\mathfrak g$ one has
	\begin{align}
	\label{eq:partial_leibnitz}
	\partial X\wedge Y=(-1)^{k+1} X\wedge \partial Y.
	\end{align}
\end{Lemma}

\proof

	Let $e_1,\ldots,e_n$ be a basis of $\mathfrak g$. By linearity we may assume that both $X$ and $Y$ are wedge products of basis elements. We will distinguish three cases.
	First, if $X, Y$ have three or more common factors, then both sides of \eqref{eq:partial_leibnitz} vanish.
	Second, if $X, Y$ have two factors in common, say $X=e_1\wedge \dots\wedge e_{k+1}$ and $Y=e_k\wedge\dots\wedge e_{n-1}$, then the only possibly non-zero summand on both sides of \eqref{eq:partial_leibnitz}  is  
	\begin{displaymath}
		[e_k, e_{k+1}]\wedge e_1\wedge\dots \wedge e_{k-1}\wedge e_k\wedge\dots\wedge e_{n-1}.
	\end{displaymath}
	It remains to consider the case when $X,Y$ have exactly one factor in common, say $X=e_1\wedge \dots\wedge e_{k+1}$ and $Y=e_{k+1} \wedge\dots\wedge e_n$. Then 
	\begin{align*}
		\partial X \wedge Y & =  (-1)^{k} \sum_{i=1}^{k} e_1 \wedge \cdots \wedge e_{i-1} \wedge [e_{k+1}, e_i] \wedge e_{i+1} \wedge \cdots \wedge e_n,\\
		X \wedge \partial Y & =  \sum_{i=k+2}^n e_1 \wedge \cdots \wedge e_{i-1} \wedge [e_{k+1},e_i] \wedge e_{i+1} \wedge \cdots \wedge e_n,
	\end{align*}
	so that 
	\begin{align*}
		&\partial X \wedge Y+(-1)^k X \wedge \partial Y \\
		&\quad = (-1)^k \sum_{i=1}^n e_1 \wedge \cdots \wedge e_{i-1} \wedge [ e_{k+1},e_i] \wedge e_{i+1} \wedge \cdots \wedge e_n\\
		&\quad =(-1)^k \tr (\ad_{ e_{k+1}}),
	\end{align*}
	which vanishes by \eqref{eq:trace} since $G$ is unimodular.
\endproof

\begin{Corollary}\label{cor:closed_form}
Let $0 \leq p \leq 2n-1$ and $\tau \in \Omega^p(\p_G)^{L_G}\otimes \largewedge^n\g$. For $0\leq k\leq p$ we have
	\begin{align}
\label{eq:dtilde}
		(\widetilde{d\tau})_k & =d(\tilde \tau_{k-1}) +(-1)^{n-p+1} \partial \tilde\tau_k.
	\end{align}
In particular, $\tau \in \Omega^n(\p_G)^{L_G} \otimes \largewedge^n\g$ is closed if and only if
	\begin{align}
\label{eq:compatible}
		d \tilde \tau_{k-1}=\partial \tilde \tau_k
	\end{align} 
for all $k$.	
\end{Corollary}
\proof
We claim that
\begin{align}
\label{eq:claim_partial}
	\partial^*\tau_k=(-1)^{n-p+k} *\partial *^{-1}\tau_k. 
\end{align}
	Applying both sides of \eqref{eq:claim_partial} to $Y \in \largewedge^{p-k+1}\mathfrak g$, the left-hand side becomes
	$$\langle \partial^*\tau_k,Y\rangle=\langle \tau_k,\partial Y\rangle.$$
	For the right-hand side we use Lemma \ref{lem:partial_leibnitz} and compute  
\begin{displaymath}
	\langle (-1)^{n-p+k} *\partial *^{-1}\tau_k,Y\rangle= (-1)^{n-p+k} \partial *^{-1}\tau_k \wedge Y= *^{-1}\tau_k \wedge \partial Y=\langle \tau_k,\partial Y\rangle,
\end{displaymath}
proving the claim.  \eqref{eq:dtilde} then follows from \eqref{eq:claim_partial} by applying $*^{-1}$ to \eqref{eq_d_and_partial}. 
\endproof



\section{Convolution of invariant differential forms}

\label{s:conv_forms}

We will now define a formal convolution product of invariant differential forms which will later be used to write an explicit formula for the convolution of smooth valuations on a general Lie group. Throughout this section, $G$ will be a unimodular Lie group of dimension $n$, $\vol$ some Haar measure,  $\mathfrak g$ its Lie algebra, $e_1,\dots,e_n$ a basis of $\mathfrak g$, $e_1^*,\ldots,e_n^*\in\g^*$ the dual basis,  and $e_1^\sharp,\dots,e_n^\sharp$ the induced vector fields on $\p_+(\g^*)$. Recall that $\Dens(\g)=\largewedge^n\g^*\otimes\ori(\g)$.

\begin{Definition}
	For $\tilde\tau,\tilde\zeta\in\Omega(\p_+(\mathfrak g^*),\largewedge^\bullet\mathfrak g)$ and $r\geq0$ we set
	\begin{displaymath}
		\hat S_r(\tilde \tau \otimes \tilde \zeta)= \sum_{|K|=r} \iota_{e_K^\sharp} \tilde\tau \wedge \iota_{e_K^*} \tilde\zeta,
	\end{displaymath}
	where the sum runs over all sets $K=\{k_1,\dots,k_r\}\subset\{1,\dots,n\}$ and we define $\iota_{e_K^\sharp}=\iota_{e^\sharp_{k_1}}\circ\cdots\circ\iota_{e^\sharp_{i_r}}$ and similarly for $\iota_{e^*_K}$.
\end{Definition}

Observe that the definition of $\hat S_r$ depends neither on the ordering of the elements in $K$ nor on the choice of a basis.

\begin{Definition} \label{def_convolution_of_forms}
	The \emph{convolution product} $\tau * \zeta \in \Omega^{p+q-n}(\p_G)^{L_G} \otimes \largewedge^n\g$ of (twisted) differential forms $\tau \in \Omega^p(\p_G)^{G \times G} \otimes \largewedge^n\g$ and $\zeta \in \Omega^q(\p_G)^{L_G} \otimes \largewedge^n\g$ is defined by 
	\begin{displaymath}
		\widetilde{(\tau * \zeta)}_j=\sum_{k+l \geq j} \epsilon^{p,q}_{k,l,j} \hat S_{k+l-j}(\tilde \tau_k \otimes \tilde \zeta_l),
	\end{displaymath}
	where 
	\begin{align*}
		\epsilon_{k,l,j}^{p,q}=(-1)^{(n+q)(n+p+l+j)+k(l+j+1)}.
	\end{align*}
	
\end{Definition}

Fixing an orientation on $G$, we can identify the Haar measure $\vol$ with a top differential form on $G$, which in turn can be both restricted to $T_eG=\g$ or pulled-back to $\mathbb P_G$. We then obtain an element
$\vol\otimes \vol^*\in \Omega^n(\p_G) \otimes \largewedge^n \g$, which depends neither on the choice of orientation, nor on the normalization of $\vol$. Explicitly,
	\begin{displaymath}
		\vol \otimes \vol^*|_{(g,[\xi])}(v_1,\ldots,v_n)=d\pi(v_1) \wedge \ldots \wedge d\pi(v_n) \in \largewedge^n T_g G \cong \largewedge^n \g.
	\end{displaymath}
We then have
\begin{equation} \label{eq_left_and_right_convolution_vol}
	(\vol \otimes \vol^*) * \zeta= \zeta\quad\text{and} \quad \tau * (\vol \otimes \vol^*)=\tau.
\end{equation}

\begin{Proposition}
	If $\tau,\zeta\in\Omega(\p_G)^{G\times G} \otimes \largewedge^n\g$ are bi-invariant, then so is $\tau*\zeta$.
\end{Proposition}

\begin{proof}
	By \eqref{eq:right-inv} and \eqref{eq:AdHodge}, it suffices to show
	$$
	\left((\Ad_{h^{-1}}^*)^*\otimes\Ad_h\right) \hat S_r(\tilde\tau\otimes\tilde\zeta)=\hat S_r(\tilde\tau\otimes\tilde\zeta), \quad r\geq 0.
	$$
	To this end, observe first that for any index subset $K$ we have
	\begin{align*}
		(\Ad_{h^{-1}}^*)^*\circ\iota_{e_K^\sharp}=\iota_{(\Ad_he_K)^\sharp}\circ(\Ad_{h^{-1}}^*)^*
	\end{align*}
	and
	\begin{align*}
		\Ad_h\circ\iota_{e_K^*}=\iota_{(\Ad_h e_K)^*}\circ\Ad_h.
	\end{align*}
	Then the claim follows easily from the invariance of $\tau$ and $\zeta$ and the fact that the definition of $\hat S_r$ is independent of the basis of $\g$. 
\end{proof}

Now we will show that the product defined in Definition \ref{def_convolution_of_forms} is associative. For the proof, the following notation will be used: Let $K=\{k_1,\dots,k_a\}$ and $L=\{l_1,\dots,l_b\}$ be strictly increasing sequences.  We define
\begin{align*}
\epsilon(K,L)=
\begin{cases}
1,&\text{if }(k_1,\dots,k_a,l_1,\dots,l_b)\text{ is an even permutation of }K\cup L,\\
0,&\text{if }K\cap L\neq\emptyset,\\
-1,&\text{if }(k_1,\dots,k_a,l_1,\dots,l_b)\text{ is an odd permutation of }K\cup L.
\end{cases}
\end{align*}
This notation will also be used in Section \ref{s:compact}.

\begin{Proposition}
	\label{prop:associativity}
The convolution product of forms over $\mathbb P_G$ is associative.
\end{Proposition}

\begin{proof}
Let $\tau\in\Omega^p(\p_G)^{G\times G} \otimes \Dens(\mathfrak g^*)$,  $\zeta\in\Omega^q(\p_G)^{G\times G} \otimes \Dens(\mathfrak g^*)$, and $\kappa\in\Omega^r(\p_G)^{L_G} \otimes \Dens(\mathfrak g^*)$. Using the definition, we compute 
		\begin{align*}
			\widetilde{(\tau * \zeta) * \kappa} &= \sum_{k,l,m} \sum_{J,K}  \epsilon_1  \iota_{e_J^\sharp}(\iota_{e_K^\sharp}\tilde\tau_k \wedge \iota_{e_K^*} \tilde\zeta_l)\wedge \iota_{e_J^*}\tilde\kappa_m\\
			& \quad= \sum_{k,l,m} \sum_{J,K} \sum_{L \subset J} \epsilon_2  \iota_{e_{L \cup K}^\sharp} \tilde\tau_k \wedge \iota_{e_{J \setminus L}^\sharp} \iota_{e_K^*} \tilde\zeta_l \wedge \iota_{e_J^*}\tilde\kappa_m\\
			&\quad = \sum_{k,l,m} \sum_{\substack{K,L,M \\ L \cap M = \emptyset \\ L \cap K = \emptyset}} \epsilon_3 \iota_{e_{L \cup K}^\sharp} \tilde\tau_k \wedge \iota_{e_M^\sharp} \iota_{e_K^*}\tilde \zeta_l \wedge \iota_{e_{L \cup M}^*}\tilde\kappa_m,
		\end{align*}
		where 
		\begin{align*}
			\epsilon_1 & = \epsilon_{k,l,k+l-|K|}^{p,q} \epsilon_{k+l-|K|,m,k+l+m-|K|-|J|}^{p+q-n,r}\\
			\epsilon_2 & = \epsilon_1 \cdot (-1)^{(k-|K|)|J \setminus L|} \epsilon(L,J\setminus L) \epsilon(L,K)\\
			\epsilon_3 & =\epsilon_{k,l,k+l-|K|}^{p,q} \epsilon_{k+l-|K|,m,k+l+m-|K|-|L|-|M|}^{p+q-n,r} \cdot (-1)^{(k-|K|)|M|}  \epsilon(L,M) \epsilon(L,K).
		\end{align*}
			Similarly, 
			\begin{align*}
				\widetilde{\tau * (\zeta * \kappa)} & = \sum_{k,l,m} \sum_{J,M}\epsilon'_1  \iota_{e_{J}^\sharp}\tilde\tau_k \wedge \iota_{e_{J}^*}( \iota_{e_{M}^ \sharp}\tilde\zeta_l \wedge \iota_{e_{M}^*} \tilde\kappa_m)\\
				& \quad= \sum_{k,l,m} \sum_{J,M} \sum_{K \subset J} \epsilon'_2   \iota_{e_{J}^\sharp}\tilde\tau_k \wedge  \iota_{e^*_{K}} \iota_{e_{M}^\sharp}\tilde\zeta_l \wedge \iota_{e^*_{J \setminus K \cup M}}\tilde \kappa_m \\
				& =\quad \sum_{k,l,m} \sum_{\substack{K,L,M\\ L \cap K=\emptyset\\L \cap M=\emptyset}} \epsilon_3'  \iota_{e_{K\cup L}^\sharp}\tilde\tau_k \wedge  \iota_{e^*_{K}} \iota_{e_{M}^\sharp}\tilde\zeta_l \wedge \iota_{e^*_{L\cup M}} \tilde\kappa_m,
			\end{align*}
			where 
			\begin{align*}
				\epsilon_1' & = \epsilon_{k,l+m-|M|,k+l+m-|M|-|J|}^{p,q+r-n} \epsilon_{l,m,l+m-|M|}^{q,r}\\
				\epsilon_2' & = \epsilon_1' \cdot (-1)^{(n-q+l)|J\setminus K|}  \epsilon(K,J\setminus K) \epsilon(J\setminus K,M)\\
				\epsilon_3' & = \epsilon_{k,l+m-|M|,k+l+m-|M|-|K|-|L|}^{p,q+r-n} \epsilon_{l,m,l+m-|M|}^{q,r} \cdot (-1)^{(n-q+l)|L|}  \epsilon(K,L) \epsilon(L,M).
			\end{align*}
			Using $\epsilon(K,L) \epsilon(L,K)=(-1)^{|K|\cdot|L|}$ it follows that $\epsilon_3= \epsilon_3'$.
		\end{proof}

\section{Compact Lie groups}
\label{s:compact}
In this section, we prove the explicit formula for the convolution of left-invariant, smooth valuations on a compact Lie group in terms of differential forms given in Theorem \ref{mthm:compact}. Throughout the section, $G$ will denote an $n$-dimensional compact Lie group with Lie algebra $\mathfrak g$.

We will consider the Haar probability measure on $G$. This choice identifies $\Dens(\g^*)$ with $\CC$. The isomorphism \eqref{eq:Hodge-1} thus becomes $*^{-1}:\largewedge^k\g^* \otimes \ori(\g) \to\largewedge^{n-k}\g$ and the convolution of invariant differential forms of Definition \ref{def_convolution_of_forms} is a map  $\Omega_{\ori}(\PP_G)^{G\times G} \times \Omega_{\ori}(\PP_G)^{L_G} \to \Omega_{\ori}(\PP_G)^{L_G}$.

We will need some maps defined in \cite{alesker_bernig_convolution}. Consider
\begin{displaymath}
	\mathcal M:=\{(g_1,g_2,[\xi_1:0]\} \cup \{(g_1,g_2,[0:\xi_2])\} \subset \p_{G \times G}
\end{displaymath}
and let $F:\hat \p \to \p_{G \times G}$ be the oriented blow-up along $\mathcal M$. Let $(G \times G) \times_{m,\pi} \p_G$ be the preimage of the diagonal under the map $m \times \pi:(G \times G) \times \p_G \to G \times G$, endowed with the preimage orientation, see \cite{guillemin_pollack}. We then have a diffeomorphism 
\begin{align} 
\begin{split} \label{eq_diffeo_fiberproduct_product}
	 (G \times G)_{m,\pi}\times \p_G & \cong G \times G \times \p_+(\mathfrak g^*), \\
	 (g,h,gh,[\xi])  & \mapsto (g,h,[ dL^*_{(gh)^{-1}}\xi]).
\end{split}
\end{align}
If we endow the space on the right-hand side with the product orientation, then this diffeomorphism is orientation preserving if and only if $n$ is even.

The differential of the multiplication map $m:G\times G\to G$ is a map
\begin{displaymath}
	dm^*:(G \times G) \times_{m,\pi} \p_G \to \p_{G \times G}
\end{displaymath}
whose image is disjoint from $\mathcal M$ by \cite[Lemma 5.2]{alesker_bernig_convolution}.  Let $p:(G \times G) \times_{m,\pi} \p_G \to \p_G$ be the projection and let $\Phi:\hat \p \to \p_G \times \p_G$ be the map given outside $\mathcal M$ by $(g_1,g_2,[\xi_1:\xi_2]) \mapsto ((g_1,[\xi_1]),(g_2,[\xi_2]))$. Since $F$ is a local diffeomorphism outside $\mathcal M$, we may define the map $r:=\Phi \circ F^{-1} \circ dm^*:(G \times G) \times_{m,\pi} \p_G \to \p_G \times \p_G$. Explicitly,
\begin{displaymath}
	r(g,h,[\xi])=((g,[R_{h^{-1}}^*\xi]),(h,[L_{g^{-1}}^*\xi)])).
\end{displaymath} 
Finally, we will consider the obvious projections $p_1,p_2: \p_G \times \p_G \to \p_G$, $\tilde p_1,\tilde p_2:G \times G \to G$, $\hat p_1:\p_G \times G \to \p_G$, and $\hat p_2:\p_G \times G \to G$. We thus have the following commutative diagram:
\begin{equation*}
	\begin{tikzcd}
	\hat \p \arrow[r,"\Phi"] \arrow[d,"F"] & \p_G \times \p_G  \arrow[dd,"\mathrm{id} \times \pi"]  \\
	\p_{G \times G} & \\
	(G \times G)\times_{m,\pi} \p_G \arrow[u,"dm^*"] \arrow[d,"p"] \arrow[uur,"r"] \arrow[r,"\cong"] & \p_G\times G\\
	\p_G         
	\end{tikzcd}
\end{equation*}

The right translation by $h$ is a diffeomorphism $R_h:G \to G$. It induces a diffeomorphism $ R_h(g,[\xi]):=(gh,[R_h^*\xi])$ on $\p_G$. For differential forms $\rho\in\Omega(G)$ and $\tau\in\Omega(\p_G)$ and a measure $\mu\in\mathcal M(G)$, the convolution from the right is defined in the usual way:
\begin{displaymath}
  \rho * \mu:=\int_G R_{g^{-1}}^*\rho d\mu(g)\quad \text{and}\quad \tau * \mu:=\int_G  R_{g^{-1}}^*\tau d\mu(g).
\end{displaymath}

We will identify the orientations on $\ori(\g)$ with left-invariant orientations on $G$, and fix such an orientation. If $\mu$ is given by the form $\omega \in \Omega^n(G)$, then $\rho*\mu$ is the push-forward of $\tilde p_1^*\rho \wedge \tilde p_2^*\omega$ under the multiplication $G \times G \to G$, and $\tau * \mu$ is the push-forward of $\hat p_1^*\tau \wedge \hat p_2^*\omega$ under the induced map $\p_G \times G \to \p_G, (g,[\xi],h) \mapsto  R_h^*(g,[\xi])$. Since the Rumin differential and pull-back of differential forms commute with diffeomorphisms, we have $D(\tau * \mu)=D\tau * \mu$ and $\pi^*(\rho * \mu)=\pi^*\rho * \mu$.

\begin{Proposition} \label{prop_forms_convolution_non_inv}
	Let $\phi=\{c_\phi,\tau_\phi\}=[[\mu_\phi,\omega_\phi]] \in \mathcal V^\infty(G)$ with  $d\omega_\phi=D\omega_\phi$ and $\psi=\{c_\psi,\tau_\psi\}=[[\mu_\psi,\omega_\psi]] \in \mathcal V^\infty(G)$ with $d\omega_\psi=D\omega_\psi$ be smooth valuations, not necessarily left- or right-invariant.   
	Then 
	\begin{align*}
		\tau_{\phi * \psi} & =(-1)^n p_*r^*(p_1^*\tau_\phi \wedge p_2^*\tau_\psi)\\
		\phi * \psi & =[[\mu_\phi * \mu_\psi,\omega_\phi * \mu_\psi+ p_* r^*(p_1^*\tau_\phi \wedge p_2^*\omega_\psi)]].
	\end{align*}
\end{Proposition}

\proof
By \cite[Proposition 5.5]{alesker_bernig_convolution}, we have $a^*\tau_{\phi * \psi} =p_*r^*(p_1^*a^*\tau_\phi \wedge p_2^*a^*\tau_\psi)$, where $a:\p_G \to \p_G$ is the antipodal map. We can define antipodal maps on $\p_G \times \p_G$ and $(G \times G) \times_{m,\pi} \p_G$ in the natural way, and then $a$ commutes with $p_1 \times p_2,r,p$. Since the push-forward depends on the orientation of the fiber, which is of dimension $n$, we have $p_*a^*=(-1)^n a^*p_*$ and the first equation follows.

 For the second equation, we suppose first that $\mu_\psi=0$ and use an idea from \cite[Proposition 5.7]{alesker_bernig_convolution}. Let $\eta:=\phi * \psi$ and $\eta':=[[0,p_*r^*(p_1^*\tau_\phi \wedge p_2^*\omega_\psi)]]$. We want to prove that $\eta=\eta'$. We have
\begin{displaymath}
	\tau_{\eta'}=dp_*r^*(p_1^*\tau_\phi \wedge p_2^*\omega_\psi)=(-1)^n p_*r^*(p_1^*\tau_\phi \wedge p_2^*\tau_\psi)=\tau_\eta.
\end{displaymath}  
Hence $\tau_{\eta-\eta'}=0$, which means that $\eta-\eta'$ is a multiple of the Euler characteristic.  The support of $\eta$ is contained in $m(\spt \phi \times \spt \psi)$ by \cite[Proposition 5.5]{alesker_bernig_convolution}, and inspecting the definition of $\eta'$ we see that also $\spt \eta' \subset m(\spt \phi \times \spt \psi)$. If $m(\spt \phi \times \spt \psi) \neq G$, then we conclude that $\eta=\eta'$. The general case follows by using a partition of unity argument to reduce to the previous case. 

Let us finally consider the case $\mu_\psi \neq 0, \omega_\psi=0$. Then $\tau_\psi=\pi^*\mu_\psi$. The map
\begin{displaymath}
	(\mathrm{id} \times \pi) \circ r:(G \times G) \times_{m,\pi} \p_G \to \p_G \times G
\end{displaymath}
is a diffeomorphism whose inverse is given by $(g,[\xi],h) \mapsto (g,h,[R_h^*\xi])$. It is orientation preserving if and only if $n$ is even, see \eqref{eq_diffeo_fiberproduct_product}, hence the pull-back and the push-forward under the inverse map differ by a factor of $(-1)^n$. It follows that 
\begin{align*}
	\tau_{\phi * \psi} & = (-1)^n p_*r^*(p_1^*\tau_\phi \wedge p_2^* \pi^*\mu_\psi) \\
	& = (-1)^n p_* ((\mathrm{id} \times \pi) \circ r)^* (\hat p_1^*\tau_\phi \wedge \hat p_2^*\mu_\psi)\\
	& = (p \circ ((\mathrm{id} \times \pi) \circ r)^{-1})_* (\hat p_1^*\tau_\phi \wedge \hat p_2^*\mu_\psi).
\end{align*}
Note that $p \circ ((\mathrm{id} \times \pi) \circ r)^{-1}(g,[\xi],h)=\tilde R_h(g,[\xi])$, hence $\tau_{\phi * \psi}=\tau_\phi * \mu_\psi$. Now we argue as above. We let $\eta=\phi * \psi$ and $\eta':=[[\mu_\phi * \mu_\psi,\omega_\phi * \mu_\psi]]$. Then $\tau_{\eta'}=(\pi^*\mu_\phi+D\omega_\phi) * \mu_\psi=\tau_\phi * \mu_\psi=\tau_\eta$. Looking at the support of $\eta-\eta'$ as above shows that $\eta=\eta'$.
\endproof

\begin{Proposition} \label{prop_convolution_inv_forms}
Let $\tau \in \Omega^p_{\ori}(\p_G)^{L_G}, \zeta \in \Omega_{\ori}^q(\p_G)^{L_G}$, and \begin{displaymath}
	\theta:=p_*r^*(p_1^*\tau \wedge p_2^*\zeta) \in \Omega_{\ori}^{p+q-n}(\p_G)^{L_G}.
\end{displaymath} 
Then  
\begin{displaymath}
	\tilde \theta_j=(-1)^{nq} \int_G \Ad_g^*\tau dg * \zeta.
\end{displaymath} 
\end{Proposition}

\begin{proof}
We assume for simplicity that $G$ admits a bi-invariant orientation, which we fix to trivialize $\ori(\g)$. We first prove the statement under the additional assumption that $\tau$ is bi-invariant. Using our identifications, the map $r$ maps $(G \times G) \times_{m,\pi} \p_+(\mathfrak g^*)$ to $G \times \p_+(\mathfrak g^*) \times G \times \p_+(\mathfrak g^*)$. Because 
\begin{displaymath}
	dm_{g,h}(L_g X,L_h Y)=\left.\frac{d}{dt}\right|_{t=0}(gE^{tX}he^{tY})=L_{gh} (\Ad_{h^{-1}}X+Y)
\end{displaymath}
holds for any $X,Y\in\mathfrak g$, one has
\begin{displaymath}
	r(g, h,gh, [\xi])=(g, [\Ad^*_{h}\xi],h, [\xi]).
\end{displaymath}

We need to prove that 
\begin{equation} \label{eq:eta_goal}
	\tilde \theta_j=(-1)^{nq} \sum_{k+l \geq j} \epsilon_{k,l,j}^{p,q} \hat S_{k+l-j}(\tilde \tau_k \otimes \tilde \zeta_l).
\end{equation} 
By left invariance, it suffices to prove \eqref{eq:eta_goal} in a point, say $(e,[\xi])$.

Let us fix a basis $e_1,\dots,e_n$ of $\g$ satisfying $e_1\wedge\cdots\wedge e_n=1$ (recall that $\largewedge^n\g$ is identified with $\ori(\g)$ via $\vol_G$).  We have $p^{-1}(e,[\xi])=\{(g,g^{-1},e,[\xi])\mid g\in G\}$ and claim that
\begin{align}
\label{eq:Tspan}
  T_{(g, g^{-1},e,[\xi])} p^{-1}(e,[\xi])=\mathrm{Span}\{(\Ad_g^{-1}e_i,-e_i,0,0)\mid i=1,\dots, n\}.
\end{align}
Here and in what follows, the tangent space of $(G\times G)\times_{m,\pi} (G\times \mathbb P_+(\mathfrak g^*))$ in a point $(g, g^{-1},e,[\xi])$ is identified with $\mathfrak g \oplus \mathfrak g \oplus \mathfrak g \oplus T_{[\xi]}\p_+(\mathfrak g^*)$ via $dL_g\oplus dL_{g^{-1}}\oplus\id\oplus\id$. To see \eqref{eq:Tspan}, consider a curve $c$ in $G$ with $c(0)=e$ and $c'(0)=e_i$. Then the curve 
\begin{displaymath}
	t \mapsto  (g \Ad_{g^{-1}}c(t),g^{-1} c(t)^{-1},e,[\xi])
\end{displaymath}
is a curve in $p^{-1}(e,[\xi])$, and its derivative at $t=0$ is $(\Ad_g^{-1}e_i,-e_i,0,0)$. The image under $dr$ of this vector is 
\begin{displaymath}
  \widehat e_i:= (\Ad_{g^{-1}}e_i,-d\Ad_{g^{-1}}^* e_i^\sharp|_{[\xi]},-e_i,0).
\end{displaymath}
Here, similarly as before, the tangent space of $G \times \p_+(\mathfrak g^*) \times G \times \p_+(\mathfrak g^*)$ in a point $(g,[\Ad^*_{g^{-1}}\xi],g^{-1},[\xi])$ is identified with $\mathfrak g\oplus T_{[\Ad^*_{g^{-1}}\xi]}\p_+(\mathfrak g^*) \oplus\mathfrak g\oplus T_{[\xi]}\p_+(\mathfrak g^*)$.

Fix $x_1,\dots,x_{p+q-n-j}\in\mathfrak g$ which we identify with left-invariant vector fields on $G$. Then  the vector
\begin{displaymath}
  \tilde x_i:=(\Ad_{g^{-1}}x_i,0,x_i,0)\in T_{(g, g^{-1},e,[\xi])}(G\times G)\times_{m,\pi} (G\times \mathbb P_+(\mathfrak g^*))
\end{displaymath}
is a lift of $x_i$, i.e., satisfies $dp(\tilde x_i)=(x_i,0)$. This can be seen by considering the curve
\begin{displaymath}
  (g \Ad_{g^{-1}}c(t),g^{-1},c(t),[\xi])\in (G\times G)\times_{m,\pi} (G\times \mathbb P_+(\mathfrak g^*))
\end{displaymath}
for a curve $c(t)\in G$ with $c(0)=e$ and $c'(0)=x_i$. We will denote
\begin{displaymath}
	\widehat x_i:=dr(\tilde x_i)=(\Ad_{g^{-1}}x_i,0,0,0).
\end{displaymath}
Further, choose any $u_1,\dots, u_j\in T_{[\xi]} \p_+(\mathfrak g^*)$. Similarly as before, $\tilde u_i=(0,0,0,u_i)$ satisfies $dp(\tilde u_i)=(0,u_i)$ and we denote
$$\widehat u_i:=dr(\tilde u_i)=(0, d\Ad _{g^{-1}}^*u_i, 0, u_i).$$
			
Let us denote $\theta^{k,l}_j:=(p_*r^*(p_1^*\tau_k \wedge p_2^*\zeta_l))_j \in\Omega^j(\mathbb P_+(\mathfrak g^*), \largewedge^{p+q-n-j} \mathfrak g^*)$ where we decompose, as usual, $\tau=\sum \tau_k$ and $\zeta=\sum\zeta_l$ but at the same time we consider $\tau_k$ and $\zeta_l$ as left-invariant forms on $G\times \mathbb P_+(\mathfrak g^*)$. Then
\begin{align*}
\Theta_{j}^{k,l}&:=\theta_j^{k,l}(u_1,\dots, u_{j})(x_1,\dots, x_{p+q-n-j})\\
&= p_*r^*(p_1^*\tau_k\wedge p_2^*\zeta_l)(u_1,\dots, u_{j},x_1,\dots, x_{p+q-n-j})\\
&= (-1)^{n(p+q-n)} \int_G p_1^*\tau_k\wedge p_2^*\zeta_l(\widehat e_1,\dots, \widehat e_n, \widehat u_1,\dots, \widehat u_{j},\widehat x_1,\dots, \widehat x_{p+q-n-j}).
\end{align*}  

 Let us call a vector $(v,w) \in \mathfrak g \oplus T_{[\xi]}\p_+(\mathfrak g^*)$ horizontal if $w=0$ and vertical if $v=0$. Then $p_2 \hat e_i$ is horizontal, $p_2\hat u_i$ is vertical and $p_2\hat x_i=0$. Since $\zeta_l$, considered as a form on $G \times \p_+(\mathfrak g^*)$, is of bidegree $(q-l,l)$, one has
\begin{align*}
\Theta_j^{k,l}
&=\sum_{I,J} \epsilon_1 \int_G p_1^*\tau_k(\widehat e_{I^c},\widehat u_{J^c},\widehat x_M)\,p_2^*\zeta_l(\widehat e_I, \widehat u_J)\\
&=\sum_{I,J} \epsilon_2 \int_G p_1^*\tau_k(\widehat e_{I^c},\widehat u_{J^c},\widehat x_M)\,p_2^*\zeta_l(\widehat u_J,\widehat e_I),
\end{align*} 
where the sum is over subsets $I\subset\{1,\dots,n\}$ and $J\subset\{1,\dots,j\}$ such that $|I|=q-l$ and $|J|=l$, and where we denote $M:=\{1\dots,p+q-n-j\}$ and $\epsilon_1:=(-1)^{n(p+q+1)} \cdot (-1)^{l(n+q+1)+pq}\epsilon(I,I^c)\epsilon(J,J^c), \epsilon_2:=(-1)^{(q-l)l} \epsilon_1$.

Similarly, $p_1\hat x_i$ is horizontal, $p_1\hat u_i$ is vertical, while $p_1\hat e_i$ is the sum of a horizontal and a vertical part: $p_1\hat e_i=(p_1\hat e_i)^h + (p_1\hat e_i)^v$. Since $\tau_k$ is of bidegree $(p-k,k)$, we obtain
\begin{align*}
\Theta_j^{k,l}&=\sum_{I,J} \sum_K \epsilon_3 \int_G \tau_k((p_1 \widehat e_K)^v,p_1 \widehat u_{J^c},(p_1 \widehat e_{I^c\setminus K})^h,p_1\widehat x_M)\,p_2^*\zeta_l(\widehat u_J,\widehat e_I),
\end{align*}
where the inner sum runs over subsets $K\subset I^c$ with $|K|=k+l-j$ and $\epsilon_3:=(-1)^{(j-l)(n-q+j-k)} \epsilon(K,I^c\setminus K) \epsilon_2$. Explicitly, plugging the vectors $\widehat e_i,\widehat u_i,\widehat x_i$, we get that $\Theta_j^{k,l}$ equals
\begin{align*}
&\sum_{I,J,K} \epsilon_4 \int_G \tau_k(d\Ad_{g^{-1}}^* e_K^\sharp,d\Ad _{g^{-1}}^*u_{J^c})(\Ad_{g^{-1}}e_{I^c\setminus K},\Ad_{g^{-1}}x_M)\,\zeta_l(u_J)(e_I)\\
&=\sum_{I,J,K} \epsilon_4 \tau_k(e_K^\sharp,u_{J^c})(e_{I^c\setminus K},x_M)\,\zeta_l(u_J)(e_I),
\end{align*}
where we used the bi-invariance of $\tau$ (which is equivalent to $\Ad$-invariance of each $\tilde \tau_k$) and put $\epsilon_4:=(-1)^{k-j+q} \epsilon_3$. 
 
Let us write $x_M:=x_1\wedge\cdots\wedge x_{n-j}$ and similarly for $e_i$. Using the identification $\largewedge^n\g\cong\mathbb C$ and  Equation \eqref{eq_star_explicit} we thus have 
\begin{displaymath}
\tilde\theta^{k,l}_j(u_1,\dots,u_j)\wedge x_M=\sum_{I,J,K} \epsilon_4  \left[\iota_{e_K^\sharp}\tilde\tau_k(u_{J^c})\wedge e_{I^c\setminus K}\wedge x_M\right]\cdot\left[\tilde\zeta_l(u_J)\wedge e_I\right].
\end{displaymath} 
Denoting $\tilde\zeta_l(u_J)=\sum_{|A|=n+l-q}\zeta_{l,A}e_A$, one has
\begin{displaymath}
\tilde\zeta_l(u_J)\wedge e_I=\zeta_{l,I^c}e_{I^c}\wedge e_I=\zeta_{l,I^c}\epsilon(I^c,I).
\end{displaymath}
Observe also that
\begin{displaymath}
i_{e_K^*}e_{I^c}=\epsilon(K, I^c\setminus K)e_{I^c\setminus K}
\end{displaymath}
whenever $K\subset I^c$ and $i_{e_K^*}e_{I^c}=0$ otherwise. Altogether, we compute
\begin{align*}
\tilde\theta^{k,l}_j(u_1,\dots,u_j)\wedge x_M & =\sum_{J,K}\epsilon_5  \iota_{e_K^\sharp}\tilde\tau_k (u_{J^c})\wedge \sum_I \zeta_{l,I^c} \cdot i_{e_K^*}e_{I^c}\wedge x_M\\
&=\sum_{J,K}\epsilon_5  \iota_{e_K^\sharp}\tilde\tau_k(u_{J^c})\wedge  i_{e_K^*}\tilde\zeta_l(u_J)\wedge x_M\\
&=\sum_{K} \epsilon_6  \left(\iota_{e_K^\sharp}\tilde\tau_k \wedge  i_{e_K^*}\tilde\zeta_l\right)(u_1,\dots,u_j)\wedge x_M,
\end{align*}
where the summations over $K$ are now over \emph{all} subsets of $\{1,\dots,n\}$ such that $|K|=k+l-j$ and $\epsilon_5=\epsilon(K,I^c \setminus K) \epsilon(I^c,I) \epsilon_4$, $\epsilon_6=\epsilon(J^c,J)\epsilon_5$. Since this holds for arbitrary vector fields $x_i$ and $u_i$, we get
\begin{align*}
\tilde\theta^{k,l}_j=\epsilon_6 \sum_{K} \iota_{e_K^\sharp}\tilde\tau_k \wedge  i_{e_K^*}\tilde\zeta_l=\epsilon_6 \hat S_{k+l-j}(\tilde \tau_k \otimes \tilde \zeta_l).
\end{align*}

Using $\epsilon(I,I^c)\epsilon(I^c,I)=(-1)^{(q-l)(n-q+l)}$ and $\epsilon(J,J^c)\epsilon(J^c,J)=(-1)^{l(j-l)}$, one easily checks that $\epsilon_6=(-1)^{nq} \epsilon_{k,l,j}^{p,q}$. Summing over $k,l$ and observing that $\tilde\theta^{k,l}_j=0$ if $k+l<j$ proves \eqref{eq:eta_goal}. 

The general case follows by averaging over $G$, noting that for left-invariant forms $\tau,\zeta$ we have $(\Ad_g^*\tau )* \zeta=\tau * \zeta$. 
\end{proof}

\proof[Proof of Theorem \ref{mthm:compact}] 
Let $\phi=\{c_\phi,\tau_\phi\} \in \mathcal V^\infty(G)^{G \times G}$ and $\psi=\{c_\psi,\tau_\psi\}=[[\mu_\psi,\omega_\psi]] \in \mathcal V^\infty(G)^{L_G}$ with $d\omega_\psi=D\omega_\psi$. By compactness, we may assume that $\omega_\psi$ is left-invariant. Combining Proposition \ref{prop_forms_convolution_non_inv} and Proposition \ref{prop_convolution_inv_forms} with $p=q=n$ we find
$$
\tau_{\phi * \psi}=(-1)^n p_*r^*(p_1^*\tau_\phi \wedge p_2^*\tau_\psi)=\tau_\phi * \tau_\psi.
$$
Similarly, using Proposition \ref{prop_forms_convolution_non_inv}, Proposition \ref{prop_convolution_inv_forms} with $p=n, q=n-1$, and $\mu_\psi=\mu(\psi) \vol_G$, where $\vol_G$ is the Haar probability measure, gives
\begin{align*}
  c_{\phi * \psi} & =\pi_*(\omega_{\phi*\psi})\\
  &  =\pi_*(\omega_\phi * \mu_\psi+p_*r^*(p_1^*\tau_\phi \wedge p_2^*\omega_\psi))\\
    & =c_\phi \mu(\psi) + \pi_* (\tau_\phi * \omega_\psi).
\end{align*}
\endproof			

Let us finish this section with a basic relation between the two natural pairings induced by the product and convolution.

\begin{Lemma}\label{lem:convolution_pairing}
	For $\phi,\psi\in \mathcal V^\infty(G)$ one has
\begin{displaymath}
	\int_G (\phi\cdot \psi)=\phi\ast \mathrm{Inv}^*\psi(e).
\end{displaymath}	
\end{Lemma}

\proof
Let $\Delta:G\to G\times G$ denote the diagonal embedding. Then 
\begin{align*}
	\phi\ast \mathrm{Inv}^*\psi(e)&= (\phi\boxtimes\mathrm{Inv}^*\psi)(m^{-1}(e))\\
	&=(\phi\boxtimes\mathrm{Inv}^*\psi)( (\id\times \mathrm{Inv})\circ \Delta(G))
	\\&=(\phi\boxtimes \psi) (\Delta(G))\\
	&= (\phi\cdot \psi)(G).
\end{align*}
\endproof

\section{Existence of bi-invariant valuations}
\label{sec:existence}

A well-known result of Milnor \cite[Lemma 7.5]{milnor76} asserts that a connected Lie group admits a bi-invariant riemannian metric if and only if it is the cartesian product of a compact Lie group and a vector space. The intrinsic volumes (or Lipschitz--Killing valuations) of such a metric then constitute examples of smooth bi-invariant valuations, as is the case for the coefficients $\phi_k$ of the Taylor series for the volume of an $\epsilon$-tube of $X\in\mathcal P(G)$, $\vol(X_\epsilon)=\sum_{k=0}^\infty \phi_k(X)\epsilon^k$. We will now show that in fact there are no other groups admitting a bi-invariant smooth valuation distinct from linear combinations of the Euler characteristic and the Haar measure.

\begin{Lemma}\label{lem:zero_implies_infinite}
	Let $G$ be any real Lie group and $V$ a finite-dimensional representation of $G$. If $v\in V$ is non-zero and $Gv$ has $0$ as its limit point, then $Gv$ is unbounded.
\end{Lemma}
\proof
Let $Z\subset V$ be the smallest $G$-invariant affine subspace containing $v$. As $0$ is a limit point of $Gv$, $Z$ is linear. Assuming $Gv$ is bounded, we conclude that the closure of the convex hull of $Gv$, denoted $K$, is a compact convex body in $Z$, which by construction has full dimension in $Z$. The John ellipsoid of the convex hull of $K \cup (-K)$ then defines a $G$-invariant euclidean norm $x\mapsto |x|$ in $Z$. Since $|v|\neq 0$, it cannot happen that $g_jv\to 0$ for a sequence $g_j\in G$, as $|g_jv|=|v|>0$. This contradiction completes the proof.
\endproof

 \begin{Lemma}\label{lem:unbounded}
 Let $G$ be any real Lie group and $V$ a finite-dimensional representation of $G$. If $V$ has no invariant euclidean structure, then there is an open dense set of vectors $v\in V$ with $Gv$ unbounded.
 \end{Lemma}

 \proof 
Take a $G$-invariant centrally symmetric convex body $K$ of maximal dimension. Then $\dim K < \dim V$, since otherwise the John ellipsoid of $K$ would define a $G$-invariant euclidean structure on $V$. Suppose that a vector $v \in V \setminus \mathrm{span} K$ has a bounded orbit. Then the convex hull of the closure of $Gv \cup K \cup (-Gv)$ is a centrally symmetric, $G$-invariant convex body of dimension strictly larger than $\dim K$, which is a contradiction. 
 \endproof

\begin{Lemma}\label{lem:away_from_zero}
		Let $G$ be a connected real Lie group and $\mathfrak g$ its Lie algebra.  Then $\Ad_G v$ does not contain $0$ in its closure for some open dense set $U$ of vectors $v\in\mathfrak g$.
\end{Lemma}
\proof
Choose a maximal proper ideal $\mathfrak h\subset\mathfrak g$, so that $\mathfrak a=\mathfrak g/\mathfrak h$ is either simple or abelian and one-dimensional. Let $\pi:\mathfrak g\to\mathfrak a$ be the projection.

In case $\mathfrak a=\R$, it holds for all $u,v\in\mathfrak g$ that $[u+\mathfrak h, v+\mathfrak h]=0$, so that $[u,v]\in\mathfrak h$. Since $G$ is connected, we have $\pi(\Ad_g v)=\pi(v)$ for all $v\in \mathfrak g$, $g\in G$, and so we may take $U=\{v\mid \pi(v)\neq 0\}=\mathfrak g \setminus \mathfrak h$.

In case $\mathfrak a$ is simple, let $B$ be its Killing form, which is non-degenerate. Since $\pi([x,v])=[\pi(x), \pi(v)]$ for all $x, v \in\g$, it holds that
$$B(\pi(\Ad_g v))=B(\Ad_g(\pi( v)))=B(\pi(v))$$
for all $g\in G$, $v\in\mathfrak g$, and therefore we may take $U=\{v\mid  B(\pi(v))\neq 0\}$.
\endproof

We can now proceed to prove the first part of Theorem \ref{mainthm_existence_biinvariant}. In what follows, $\vol_G$ denotes some non-zero left-invariant Haar measure on $G$.

\begin{Theorem}\label{prop:no_biinvariant}
	Let $G$ be a connected real Lie group of dimension at least $2$. Then $G$ admits a smooth bi-invariant valuation $\phi\notin\Span\{\chi, \vol_G\}$ if and only if $G$ admits a bi-invariant riemannian metric. 
\end{Theorem}
\proof
If $G$ admits a bi-invariant metric, its first intrinsic volume $\mu_1$ is a bi-invariant smooth valuation as desired.

Assume now that $G$ does not admit a bi-invariant metric, and let $\phi\notin\Span\{\chi, \vol_G\}$ be a non-zero bi-invariant valuation. Assume first that $G$ is connected. We proceed to arrive at a contradiction.

Subtracting a multiple of $\chi$, and of $\vol_G$ if $G$ is unimodular, we may assume $\phi\in\mathcal W_j(G)\setminus\mathcal W_{j+1}(G)$ for some $1\leq j\leq n-1$. Any fixed choice of orientation on $G$ is evidently invariant under the adjoint action, yielding an equivariant identification $\Dens(\mathfrak g)=\wedge^n\mathfrak g^*$. The principal symbol of $\phi$ is an invariant element of $\Gamma(G, \Val^\infty_j(TG))$, and its value at the unit element defines a non-zero valuation $\psi\in\Val_j^\infty(\mathfrak g)$ invariant under the adjoint action of $G$. Applying the Alesker--Fourier isomorphism, we get a non-zero element
$\mathbb F\psi \in \Val_{n-j}^\infty(\mathfrak g^*)\otimes \Dens(\mathfrak g)=\Val_{n-j}^\infty(\mathfrak g^*)\otimes \wedge^n\mathfrak g^*$.

The $n$-form $\tau=\tau_{\mathbb F \psi}$ of $\mathbb F\psi$ is a non-zero $\mathrm{Ad}_G$-invariant element of 
\begin{align*}
	\Omega^j(\mathbb P_+(\mathfrak g), \wedge^{n-j} \mathfrak g)\otimes \wedge^n\mathfrak g^* & =\Gamma(\mathbb P_+(\mathfrak g),\wedge^j(\xi^*\otimes \mathfrak g/\xi)^*\otimes\wedge^{n-j} \mathfrak g)\otimes \wedge^n\mathfrak g^*\\
	& = \Gamma(\mathbb P_+(\mathfrak g),\wedge^j(\xi \otimes \xi^\perp) \otimes\wedge^{n-j} \mathfrak g)\otimes \wedge^n\mathfrak g^*\\
	& =  \Gamma(\mathbb P_+(\mathfrak g),\xi^{\otimes j}\otimes \wedge^j \xi^\perp\otimes\wedge^j \mathfrak g^*),
\end{align*} 
where in the last step we used the Hodge star isomorphism $\ast:\wedge^{n-j} \mathfrak g \otimes \wedge^n\mathfrak g^*\to \wedge^{j} \mathfrak g^*$. Since $\tau$ is closed and vertical we have in fact
\begin{displaymath}
	\tau \in   \Gamma(\mathbb P_+(\mathfrak g),\xi^{\otimes j}\otimes \mathrm{Sym}^2(\wedge^j \xi^\perp)),
\end{displaymath}
see \cite[Lemma 4.8]{faifman_wannerer_fourier}.

Let $\eta\in \mathbb P_+(\mathfrak g)$ be a line such that $\mathrm{Ad}_G v$ is unbounded and $0\notin\overline{\mathrm{Ad}_G v}$ for any non-zero $v\in \eta$. Those assumptions hold by Lemmas \ref{lem:unbounded} and 
\ref{lem:away_from_zero} for a dense set of lines $\eta$.

Write $\tau_\eta=c(\eta) v^{\otimes j}\otimes Q$, where $v\in\eta, v \neq 0$,  $Q\in  \mathrm{Sym}^2(\wedge^j \eta^\perp)$, $Q\neq 0$, and $c(\eta)\in\R$. Assume $c(\eta)\neq 0$.

Take a sequence $g_i\in G$ such that $\mathrm{Ad}_{g_i}v\to \infty$. By the smoothness and $\Ad_G$-invariance of $\tau$, it must hold that $\mathrm{Ad}_{g_i}Q\to 0$.

By Lemma \ref{lem:zero_implies_infinite}, there is a sequence $h_i\in G$ such that $\mathrm{Ad}_{h_i}Q\to\infty$ in $\Sym^2(\wedge^j\mathfrak g^*)$. Again since $\tau$ is smooth, $\mathrm{Ad}_{h_i}v\to 0$, a contradiction.
It follows that $c(\eta)=0$ for a dense set of $\eta$, and so $\tau=0$. This contradiction concludes the proof for connected $G$.
\endproof

To complete the proof of Theorem \ref{mainthm_existence_biinvariant}, we will need two more auxiliary statements.

\begin{Proposition}  \label{prop_finite_dimensionality}
	Let $G$ be a connected compact Lie group of dimension at least $2$. The following statements are equivalent:
	\begin{enumerate}
		\item $\dim \mathcal V^\infty(G)^{G \times G}<\infty$.
		\item The coadjoint action of $G$ on $\p_*(\mathfrak g^*)$ is transitive. 
		\item $G$ is isomorphic to either $S^3$ or $\mathrm{SO}(3)$.
	\end{enumerate}
\end{Proposition}

\proof

(i)$\implies$(ii): Suppose that the coadjoint action is not transitive on $\p_+(\mathfrak g^*)$. By \cite[Theorem 4.1]{bernig_aig10} the space of $\Val_1(\mathfrak g)^{\Ad G}$ is infinite-dimensional. Using the results and the notation from \cite[Section 1.5]{alesker_val_man4}, we can extend an element of this space as a bi-invariant section of $\underline{\Val}_1(TG)=W_1(G)/W_2(G)$. We thus find a smooth valuation in $W_1(G)$ whose image in $\underline{\Val}_1(TG)$ is the given section. By averaging with respect to the left and right translations on $G$, we finally get a bi-invariant valuation in $W_1(G)$ whose image in $\Val_1(\mathfrak g)$ is the element we started with. Hence the subspace $W_1(G)^{G \times G} \subset \mathcal V^\infty(G)^{G \times G}$ is infinite-dimensional. 

(ii)$\implies$(i): Under the assumption (ii), $G$ admits a bi-invariant 
riemannian metric such that $G \times G$ acts transitively on the unit tangent bundle. It is well known that the space of smooth invariant valuations is finite-dimensional in this case, see for instance \cite[Section 10]{alesker_bernig}.  

(iii)$\implies$(ii): This is clear. 

(ii)$\implies$(iii): As above, there exists a bi-invariant metric on $G$. By assumption, the adjoint action of $G$ is transitive on the unit sphere in $\mathfrak g$. Let $H \subset G$ be the kernel of the adjoint action. Then we obtain an effective and transitive action of $G'=G/H$ on the unit sphere in $V=\mathfrak g$.

The connected groups $G'$ that act transitively and effectively on the unit sphere in some euclidean vector space $V$ are known \cite{borel49, montgomery_samelson43}. The possible pairs $(V,G')$ are 
\begin{gather*}
	(\R^n,\mathrm{SO}(n)), (\C^n,\mathrm U(n)), (\C^n,\mathrm{SU}(n)),\\
	 (\mathbb H^n,\mathrm{Sp}(n)),	(\mathbb H^n,\mathrm{Sp}(n) \cdot \mathrm{U}(1)),  (\mathbb H^n,\mathrm{Sp}(n) \cdot \mathrm{Sp}(1)), 
\end{gather*}
for $n\in\mathbb N$, and the three exceptional cases 
\begin{displaymath}
	(\R^7,\mathrm G^2), (\R^8,\mathrm{Spin}(7)), (\R^{16},\mathrm{Spin}(9)).
\end{displaymath}
In our situation, we have $\dim V=\dim G \geq \dim G'=\dim G-\dim H$ and $\dim G \geq 2$. Noting that $\mathrm{Sp}(1) \cdot \mathrm{U}(1) \cong \mathrm U(2)$ and $\mathrm{SU}(2) \cong \mathrm{Sp}(1)$, the only pairs from the list satisfying these dimension bounds are 
\begin{displaymath}
	(\R^3,\mathrm{SO}(3)), (\C^2,\mathrm U(2)), (\C^2,\mathrm{SU}(2)).
\end{displaymath}

In the case $(\R^3,\mathrm{SO}(3))$, $G$ is a finite cover of $\mathrm{SO}(3)$. Since $\pi_1(\mathrm{SO}(3))=\mathbb Z/2\mathbb Z$, there are only the two possibilities $G=\mathrm{SO}(3)$ and $G=S^3$. 

It remains to exclude the two other cases. The kernel of the action of $G \times G$ on $G$ is the diagonal embedding of $H$, hence $\tilde G:= (G \times G)/H$ acts transitively and effectively on the unit tangent bundle, i.e., $(G,\tilde G)$ is a compact isotropic manifold of dimension $4$.  

Isotropic manifolds $(M,\tilde G)$ are classified in \cite{tits55, wang52}. The only compact examples of dimension $4$ are $M=\mathbb{RP}^4, \mathbb{CP}^2,\mathbb{HP}^1$, and in all three cases $\tilde G$ is the identity component of the respective isometry group. Since $\mathbb{RP}^4$ is non-orientable, it cannot be diffeomorphic to a Lie group. On a compact Lie group $G$, the vanishing of the first cohomology implies the vanishing of the second cohomology and the non-vanishing of the third cohomology group \cite[Corollaries 12.9 and 12.11]{bredon93}, hence $\mathbb{CP}^2$ and $\mathbb{HP}^1 \cong S^4$ cannot be diffeomorphic to Lie groups either. 
\endproof

\proof[Proof of Theorem \ref{mainthm_existence_biinvariant}]
The first statement is just Theorem \ref{prop:no_biinvariant}. For the second statement, let us assume that $G$ admits a bi-invariant riemannian metric. By Milnor's result \cite{milnor76} we have $G \cong K \times \R^m$, where $K$ is a compact Lie group and $m \geq 0$. If $m=0$, then we may apply Proposition \ref{prop_finite_dimensionality} to deduce that $\dim \mathcal V^\infty(G)^{G \times G}<\infty$ if and only if $G=K$ equals $S^3$ or $\mathrm{SO}(3)$.
If $m\geq 1$, consider the covering map $q:K\times\R^m\to K\times T$ where $T=\R^m/\mathbb Z^m$, and note that the pull-back map \[q^*:\mathcal V^\infty(K\times T)^{(K\times T)\times (K\times T)}\to  \mathcal V^\infty(K\times \R^m)^{(K\times \R^m)\times (K\times \R^m)}\] is an isomorphism. The former space is infinite-dimensional by the case of $m=0$ with $K\times T$ replacing $K$.
\endproof

\section{Convolution of smooth bi-invariant valuations on Lie groups}
\label{sec:convolution_unimodular}

In this section we define the convolution of bi-invariant smooth valuations on a general unimodular Lie group, thus unifying the Alesker--Bernig convolution on compact Lie groups with the Bernig--Fu convolution of translation-invariant valuations on a linear space.

\subsection{Haar measures and forms}

We first verify that for a unimodular Lie group $G$, the line $\mathcal M(G)^G$ of Haar measures is naturally a direct summand of the space of left-invariant valuations. To this end,  we define the space of \emph{measure-free} valuations $\mathcal V_{mf}^{\infty}(G)^{L_G} \otimes \Dens(\mathfrak g^*)$ to consist of all valuations $\{c,\tau\}{ \in\mathcal V^{\infty}(G)^{L_G} \otimes \Dens(\mathfrak g^*)}$ for which $\tau_0=0$.

\begin{Proposition}\label{prop:haar_direct}
Let $G$ be a unimodular Lie group. Then
	\begin{displaymath}
		\mathcal V^{\infty}(G)^{L_G} \otimes \Dens(\mathfrak g^*)=\C \oplus \mathcal V_{mf}^{\infty}(G)^{L_G} \otimes \Dens(\mathfrak g^*).
	\end{displaymath}
\end{Proposition}

\proof
Consider $\tau_0\in\Omega^0(\mathbb P_+(\mathfrak g^*), \wedge^n\mathfrak g^* \otimes\ori(\g)) \otimes \Dens(\mathfrak g^*) \cong C^{\infty}(\mathbb P_+(\mathfrak g^*))$ for any left-invariant valuation $\phi=\{c,\tau\}$. By Corollary \ref{cor:closed_form}, $d\widetilde \tau_0=\partial\widetilde \tau_1=0$. Thus $\tau_0$ is a constant which will be denoted by $\mu(\phi)$. Depending on the context, we can think of $\mu(\phi)$ as an element of $\C$ or of $\mathcal V^\infty(G)^{L_G} \otimes \Dens(\mathfrak g^*)$. Writing $\phi=\mu(\phi)+(\phi-\mu(\phi))$ concludes the proof.
\endproof

We remark that when restricted to bi-invariant valuations, the projection $\mu:\mathcal V^{\infty}(G)^{G \times G} \otimes \Dens(\mathfrak g^*) \to \C$ is a character with respect to the convolution defined below, i.e., it satisfies $\mu(\phi * \psi)=\mu(\phi) \mu(\psi)$, see Proposition \ref{pro:mu_character}.

\begin{Lemma} \label{lemma_invariant_primitive}
	Let $G$ be a unimodular Lie group. If $\tau \in \Omega^p(\p_G)^{L_G} \otimes \largewedge^n\g$ satisfies $d\tau=0$ and $\pi_*\tau=0$, then there exists $\omega \in \Omega^{p-1}(\p_G)^{L_G} \otimes \largewedge^n\g$ such that 
	$$d\omega=\tau-\tau_0\quad\text{and}\quad \pi_*\omega=0.$$
\end{Lemma}

\begin{proof}
	We will inductively construct forms $\tilde \omega_k \in \Omega^k(\p_+(\mathfrak g^*),\largewedge^{n-p+k+1}\mathfrak g)$ such that $\omega=\sum_{k=0}^{n-1} *\widetilde\omega_k$ has the desired properties. We have $\tau=\sum_k \tau_k$ with $\tau_k \in \Omega^k(\p_+(\mathfrak g^*),\largewedge^{p-k}\mathfrak g^*) \otimes   \largewedge^n\g$ and $\tilde \tau_k \in \Omega^k(\p_+(\mathfrak g^*),\largewedge^{n-p+k} \mathfrak g)$. By Corollary \ref{cor:closed_form}, the equation $d\omega=\tau-\tau_0$ is equivalent to the system of equations
	\begin{displaymath}
		d\tilde \omega_k+(-1)^{n-p} \partial \tilde \omega_{k+1}=\tilde \tau_{k+1}, \quad 0 \leq k \leq n-1.
	\end{displaymath}
	First, for $k=n-1$ we set $\tilde \omega_{n-1}=0$. Now suppose that we have already constructed $\tilde \omega_{k+1}$ for some $k \leq n-2$. Since  $d\tau=0$, Corollary  \ref{cor:closed_form}  implies
	\begin{align*}
		d(\tilde \tau_{k+1}-(-1)^{n-p}\partial \tilde \omega_{k+1})&=d\tilde \tau_{k+1}- (-1)^{n-p}\partial(\tilde \tau_{k+2}-(-1)^{n-p}\partial \tilde \omega_{k+2})\\&=d\tilde \tau_{k+1}+(-1)^{n-p+1}\partial\tilde \tau_{k+2}\\
		&=(\widetilde {d\tau})_{k+2}\\
		&= 0.
	\end{align*}
	We claim that the closed form $\tilde \tau_{k+1}-(-1)^{n-p}\partial \tilde \omega_{k+1}\in\Omega^{k+1}(\p_+(\mathfrak g^*),\largewedge^{n-p+k+1}\mathfrak g)$ is actually exact. Indeed, if $k<n-2$, then $H^{k+1}(\p_+(\mathfrak g^*))=H^{k+1}(S^{n-1})=0$. If $k=n-2$, then $H^{n-1}(\p_+(\mathfrak g^*))=H^{n-1}(S^{n-1})=\R$, but we moreover have
	\begin{displaymath}
		\int_{\p_+(\mathfrak g^*)} (\tilde \tau_{n-1}-(-1)^{n-p}\partial \tilde \omega_{n-1})=\int_{\p_+(\mathfrak g^*)} \tilde \tau_{n-1}=0
	\end{displaymath}
 since $\pi_*\tau=0$ and $\tilde \omega_{n-1}=0$. It follows that $\tilde \tau_{k+1}-(-1)^{n-p}\partial \tilde \omega_{k+1}=d\tilde \omega_k$ for some $\tilde \omega_k \in \Omega^k(\p_+(\mathfrak g^*),\largewedge^{n-p+k+1}\mathfrak g)$ which finishes the induction step. We thus have $d\omega=\tau-\tau_0$. Since $\tilde \omega_{n-1}=0$, we also get $\pi_*\omega=0$. 
\end{proof}

\begin{Proposition} 
	\label{prop_exact}
	Let $G$ be a unimodular Lie group. Every left-invariant valuation $\phi \in \mathcal V^\infty(G)^{L_G} \otimes \Dens(\mathfrak g^*)$ can be represented as $\phi=[[\sigma,\omega]]$ with $\sigma \in \Omega^n_{\ori}(G)^{L_G} \otimes \Dens(\mathfrak g^*)$ and $\omega \in \Omega^{n-1}_{\ori}(\p_G)^{L_G} \otimes \Dens(\mathfrak g^*)$ such that $d\omega=D\omega$. Moreover, for each such choice of $\sigma,\omega$, it holds that $\mu(\phi)=\sigma$.
\end{Proposition}

\proof
Let $\phi=\{\zeta,\tau\}$ with $\zeta \in C^\infty(G)^{L_G} \otimes \Dens(\mathfrak g^*), \tau \in \Omega^n_{\ori}(\p_G)^{L_G} \otimes \Dens(\mathfrak g^*)$. The Euler characteristic $\chi$ can be written as $[[\sigma',\omega']]$ with $\sigma',\omega'$ left invariant (just use Chern's construction with a left-invariant riemannian metric). Then $D\omega'+\pi^*\sigma'=0,\pi_*\omega'=1$. Subtracting a multiple of $\chi$ we may thus assume that $\zeta=0$. We take $\omega$ as in Lemma \ref{lemma_invariant_primitive} with $d\omega=\tau-\tau_0$, and $\sigma$ such that $\pi^*\sigma=\tau_0$. Then $d\omega=D\omega$, $D\omega+\pi^*\sigma=\tau$, and $\pi_*\omega=0$, hence $\phi=[[\sigma,\omega]]$.

For the last part observe that by Lemmas \ref{lem:closed_compatible} and \ref{lem:partial_leibnitz}, we have $(d\omega)_0=0$, and so $\pi^*\mu(\phi)=\tau_0$ must equal $\pi^*\sigma$.
\endproof

\subsection{From convolution on compact Lie groups to general convolution}
\label{sec:general_convolution}

	By Theorem \ref{prop:no_biinvariant}, a connected Lie group $G$ admits non-obvious bi-invariant smooth valuations if and only if $G=K\times V$ with $K$ compact and $V$ linear. As explained in the introduction, instead of proving Theorem \ref{mthm:unimodular} directly, we first define an a-priori distinct convolution operation for groups of the form $K\times V$. Thus throughout this subsection, we will reserve the notation $\phi\ast\psi$ for this distinct convolution of Lemma \ref{lem:lattices} below. It will be seen in Proposition \ref{prop_convolution_formula} to coincide with that of Definition \ref{mdef:convolution}.

\begin{Lemma}\label{lem:lattices}
	For $G=K\times V$ with $K$ compact and $V$ a linear space, fix a lattice $\Gamma\subset V$, and let $T_\Gamma=V/\Gamma$ be the corresponding compact torus. Denote $G_\Gamma=K\times T$, and let $q_\Gamma:G\to G_\Gamma$ be the natural covering map, inducing the pull-back map $q_\Gamma^*:\mathcal V^\infty(G_\Gamma)\to \mathcal V^\infty(G)$. Let $\sigma_\Gamma\in\Dens(V^*)$ be the dual density defined by $\Gamma$, which we identify with $\sigma_\Gamma \in\Dens(\mathfrak g^*)$ through the Haar probability  measure on $K$. Define a convolution operation $\ast_\Gamma$ on  $\mathcal V^\infty(G)^{G\times G}\otimes\Dens(\mathfrak g^*)$ 
	through the isomorphism
\begin{equation*}
	Q_\Gamma:=  q_\Gamma^*\otimes \sigma_\Gamma:\mathcal V^\infty(G_\Gamma)^{T_ \Gamma}\to \mathcal V^\infty(G)^{V}\otimes\Dens(\mathfrak g^*).
\end{equation*}
Then $\ast_\Gamma$ on $ \mathcal V^\infty(G)^{G\times G}\otimes\Dens(\mathfrak g^*)$ is independent of $\Gamma$. Similarly, the induced product 
\[\ast_\Gamma: \mathcal V^\infty(G)^{G\times G}\otimes\Dens(\mathfrak g^*)\times  \mathcal V^\infty(G)^{L_G}\to  \mathcal V^\infty(G)^{L_G}\]
is independent of $\Gamma$.
\end{Lemma}
\proof
Let $\Gamma, \Gamma'$ be lattices. For the first statement we ought to show that $Q:=Q_{\Gamma'}^{-1}\circ Q_\Gamma:\mathcal V^\infty(G_{\Gamma})^{G_\Gamma\times G_\Gamma}\to\mathcal V^\infty(G_{\Gamma'})^{G_{\Gamma'}\times G_{\Gamma'}}$ is an isomorphism of convolution algebras. Denoting $\delta=\frac{\sigma_\Gamma}{\sigma_{\Gamma'}}$ and $q=(q_{\Gamma'}^*)^{-1}\circ q_\Gamma^*:\mathcal V^\infty(G_\Gamma) \to \mathcal V^\infty(G_{\Gamma'})$, we have $Q=\delta q$. 
	
Under the natural identifications, we have $\tau_{q\phi}=\tau_\phi, \omega_{q\phi}=\omega_\phi,c_{q \phi}=c_\phi$ and similarly for $\psi$. However, the isomorphism \eqref{eq:Haar_identification} makes use of the Haar measure. Since $q$ maps the Haar probability measure $\nu_\Gamma$ to $\delta^{-1}\nu_{\Gamma'}$, we get $\tilde \tau_{q \phi}=\delta^{-1} \tilde \tau_\phi, \tilde \omega_{q \phi}=\delta^{-1}\tilde \omega_\phi$ etc. By Definition \ref{def_convolution_of_forms} we have
\begin{displaymath}
	\tau_{q\phi}* \tau_{q \psi}=\delta^{-1} \tau_\phi * \tau_\psi, \quad \tau_{q\phi} * \omega_{q\psi}=\delta^{-1} \tau_\phi * \omega_\psi.
\end{displaymath}  
Clearly 
\begin{displaymath}
	\mu(q\psi)=\int_{G_{\Gamma'}}\psi=\delta^{-1} \int_{G_\Gamma}\psi=\delta^{-1} \mu(\psi).
\end{displaymath} 	Theorem \ref{mthm:compact} shows that $q\phi*q\psi=\delta^{-1}q(\phi * \psi)$, which implies $Q(\phi*\psi)=Q\phi * Q\psi$ as required. 

This completes the proof of the first statement. The proof of the second statement is completely analogous.
\endproof

We conclude that Definition \ref{mdef:convolution} is independent of the choice of lattice $\Gamma$.  To prove the independence of the decomposition into the compact and linear part, we first observe that the formula in terms of differential forms established by Theorem \ref{mthm:compact} holds verbatim in the general case $G=K\times V$.
 
 \begin{Proposition} 
 \label{prop_convolution_formula}
	Let $G=K\times V$ with $K$ compact and $V$ a linear space. Let $c_\phi,c_\psi \in \Dens(\mathfrak g^*), \tau_\phi\in \Omega^n_{\ori}(\p_G)^{G \times G} \otimes \Dens(\mathfrak g^*)$,  $\tau_\psi\in \Omega^n_{\ori}(\p_G)^{L_G}$. One has 
	\begin{equation} \label{eq_convolution_formula}
		\{c_\phi,\tau_\phi\}*\{c_\psi,\tau_\psi\}=\{c_\phi\mu(\psi)+\pi_*(\tau_\phi*\omega_\psi),\tau_\phi*\tau_\psi\},
	\end{equation}
	where $\omega_\psi \in \Omega^{n-1}_{\ori}(\p_G)^{L_G}$ is such that $d\omega_\psi=\tau_\psi- \pi^*\mu_\psi$  and $\pi_*\omega_\psi=c_\psi$.
\end{Proposition}

\proof
By Proposition \ref{prop_exact}, a form $\omega_\psi$ with the given properties exists. The validity of Equation \eqref{eq_convolution_formula} then follows at once from Theorem \ref{mthm:compact} and Lemma \ref{lem:lattices}.
\endproof

 \begin{Corollary}\label{cor:product_independent}
 	Let $G$ be isomorphic to the cartesian product of a compact group and a linear space. Then the convolution on $\mathcal V^\infty(G)^{G\times G}\otimes\Dens(\mathfrak g^*)$ only depends on the Lie group structure on $G$, that is, it is independent of the decomposition $G=K\times V$. In particular, $F(\phi\ast \psi)=F\phi\ast F\psi$ for any $F\in\mathrm{Aut}(G)$.
 \end{Corollary}
 \proof
 By Proposition \ref{prop_convolution_formula}, the convolution is given by Equation \eqref{eq_convolution_formula}, which only depends on the group structure of $G$. 
 \endproof
 
It is straightforward to extend the convolution of Lemma \ref{lem:lattices} to any unimodular Lie group that admits non-obvious bi-invariant valuations. Let us now take this step.
 
 \begin{Lemma} \label{lemma_reduce_to_connected}
 		Let $G$ be a Lie group and $G_0$ the identity component. Then the natural restriction map 
 		\begin{displaymath}
 			\mathcal V^\infty(G)^{G \times G} \to \mathcal V^\infty(G_0)^{G_0 \times G_0}
 		\end{displaymath}
 		is  injective. Its image consists of all $\mathrm{Ad}_G$-invariant $\phi\in\mathcal V^\infty(G_0)^{G_0 \times G_0}$.
 	\end{Lemma}	
 	
 	\proof
 	If the restriction of a bi-invariant valuation to $G_0$ vanishes, then by bi-invariance the restriction to every connected component of $G$ vanishes and hence the valuation itself is zero. This shows injectivity of the restriction map. 
 	
 	For $\phi\in\mathcal V^\infty(G_0)^{G_0 \times G_0}$, define $\psi\in \mathcal V^\infty (G)$ by setting $\psi|_{xG_0}=L_{x}^*\phi$, which is well-defined and left-invariant by the left-invariance of $\phi$. Clearly $\psi$ restricts to $\phi$. It is then easy to see that $\psi$ is right-invariant if and only if $\phi$ is $\mathrm{Ad}_G$-invariant.
 	 	\endproof 
 	
 It is easy to construct explicit examples where the restriction is not surjective. For instance, if $G=\R^2\rtimes \{\pm 1\}$, where $-1$ acts on $\R^2$ by $-\id$, then $\mathcal V^\infty(G_0)^{G_0 \times G_0}$ consists of all translation-invariant valuations, while $\mathcal V^\infty(G)^{G \times G}$ corresponds to the even valuations.

 \begin{Proposition}\label{prop:restrict_connected}
 	Let $G$ be a unimodular Lie group  that admits a non-obvious bi-invariant valuation, and let $G_0\subset G$ be the identity component. Then $G_0=K\times V$ for some compact Lie group $K$ and a linear space $V$, and the subspace of  $\mathcal V^\infty(G_0)^{G_0 \times G_0}\otimes\Dens(\mathfrak g^*)$
 	corresponding to the image of the restriction map 
 	\begin{displaymath}
 		\mathcal V^\infty(G)^{G \times G} \to \mathcal V^\infty(G_0)^{G_0 \times G_0}
 	\end{displaymath}
 	is closed under convolution.
 \end{Proposition}
 \proof
The first statement follows  at once from Lemma \ref{lemma_reduce_to_connected} and Theorem \ref{mainthm_existence_biinvariant}. As $\mathbf{Ad}_g$ is an automorphism of $G_0$ for all $g\in G$, we have by Corollary \ref{cor:product_independent} that $\mathbf{Ad}_g^*(\phi\ast \psi)=(\mathbf{Ad}_g^*\phi)\ast (\mathbf{Ad}_g^*\psi)$ for $\phi,\psi\in\mathcal V^\infty(G_0)^{G_0\times G_0}\otimes\Dens(\mathfrak g^*)$. Thus the $\mathrm{Ad}_G$-invariant subspace of $ \mathcal V^\infty(G_0)^{G_0 \times G_0}$, which by Lemma \ref{lemma_reduce_to_connected} is the image of the restriction, is closed under convolution.
 \endproof

We conclude that the convolution on $\mathcal V^\infty(G)^{G \times G}\otimes\Dens(\mathfrak g^*)$ is  well defined for any unimodular Lie group $G$ admitting  non-obvious bi-invariant smooth valuations and, by Proposition \ref{prop_exact} it is given by Equation  \eqref{eq_convolution_formula}.  Similarly one defines the module structure
\[\ast: \mathcal V^\infty(G)^{G\times G}\otimes\Dens(\mathfrak g^*)\times  \mathcal V^\infty(G)^{L_G}\to  \mathcal V^\infty(G)^{L_G}\]
in this case. If $G$ is unimodular with $\mathcal V^\infty(G)^{G\times G}=\Span\{\chi, \vol\}$, one immediately verifies using Proposition \ref{prop_exact} and $\chi=\{1,0\}$ that Definition  \ref{mdef:convolution} gives rise to a well-defined convolution operation satisfying $(\vol\otimes \vol^*)\ast\psi=\psi$ and $(\chi\otimes \vol^*)\ast \psi=(\mu(\psi) \vol^*)\chi$ for all $\psi\in\mathcal V^\infty(G)^{L_G}$.

\subsection{Convolution of smooth valuations}

We will now prove various properties of the convolution on valuations on unimodular Lie groups given by Definition \ref{mdef:convolution}. 

\begin{Proposition}\label{pro:mu_character}
Let $G$ be a unimodular Lie group. For any two valuations $\phi,\psi\in\mathcal V^\infty(G)^{G \times G} \otimes \Dens(\mathfrak g^*)$ it holds that $\mu(\phi\ast \psi)=\mu(\phi)\mu(\psi)$.
\end{Proposition}
\proof
By Definitions \ref{mdef:convolution} and \ref{def_convolution_of_forms}, we have
\begin{displaymath}
	\widetilde{(\tau_{\phi * \psi})}_0=	\widetilde{(\tau_{\phi} * \tau_{\psi})}_0=\sum_{k,l \geq 0} (-1)^{k(l+1)} \hat S_{k+l}((\tilde \tau_\phi)_k \otimes (\tilde \tau_\psi)_l).
\end{displaymath}
Since $(\tilde \tau_\phi)_k \in \Omega^k(\p_+(\mathfrak g^*),\largewedge^k \mathfrak g)$, all terms with $l>0$ vanish. Similarly, since $(\tilde \tau_\psi)_l \in \Omega^l(\p_+(\mathfrak g^*),\largewedge^l \mathfrak g)$, all terms with $k>0$ vanish. The only remaining term is for $k=l=0$ and hence
$$(\widetilde{\tau_ {\phi * \psi}})_0=\hat S_0((\tilde \tau_\phi)_0 \otimes (\tilde \tau_\psi)_0)=(\tilde \tau_\phi)_0 \cdot (\tilde \tau_\psi)_0.$$
\endproof

\begin{Lemma} \label{lemma_inversion}
	Let $G$ be a unimodular Lie group and let $\Inv:G \to G$ be the inverse map. Then for any $\phi,\psi \in \mathcal V^\infty(G)^{G\times G} \otimes\Dens(\mathfrak g^*)$ we have 
	\begin{equation}\label{eq:inverse_convolution}
		\Inv^*(\phi * \psi)=\Inv^*\psi * \Inv^*\phi. 
	\end{equation}
\end{Lemma}

\proof 
We may assume that $G$ admits non-obvious bi-invariant valuations. Assume first that $G$ is compact. Denoting $S:G\times G\to G\times G$ the involution $S(x,y)=(y,x)$  and by $m:G \times G \to G$ the multiplication map, we have an equality of maps
\[ \Inv\circ m= m\circ S\circ (\Inv\times \Inv): G\times G\to G.\]
Pushing-forward $\phi\boxtimes \psi$ using both presentations of this map and noting that $\Inv_*=\Inv^*$ yields \eqref{eq:inverse_convolution}. The general case then follows at once by Lemma \ref{lem:lattices}, Equation \eqref{eq_convolution_formula}, and Proposition \ref{prop:restrict_connected}.
\endproof

Let us now proceed with the proofs of Theorem \ref{mthm:unimodular} and Proposition \ref{mthm:compatibility}. We will need the following statement.
 
\begin{Lemma}\label{lem:translation_surjective}
	Let $M$ be a compact smooth manifold and $V$ a real vector space of finite dimension. Consider the obvious action of $V$ on $X:=M\times V$. Then  the restriction of the quotient map $q_V:\mathcal W_k(X)\to \Gamma(X,\Val_k^\infty(TX))$ to the $V$-invariant elements, namely $q_V:\mathcal W_k(X)^V\to \Gamma(X,\Val_k^\infty(TX))^V$, is onto.	
	Furthermore, fix a Lebesgue measure on $V$ and consider the averaging map $A_1:\mathcal V^\infty_{c}(X)\to 	\mathcal V^\infty(X)^V$ given by $A_1\phi=\int_V x^*\phi dx$. Denote $\mathcal W_{k,c}(X):=\calW_k(X)\cap \calV_c^\infty(X)$. Then the restriction $A_1:\mathcal W_{k,c}(X) \to \mathcal W_k(X)^V$ is onto. In particular $A_1$ itself is onto.
\end{Lemma}

\proof 
Recall that the natural map $\mathcal W_k(X)\to \Gamma(X,\Val_k^\infty(TX))$ is onto.  By \cite[Proposition 6.2.1]{alesker_val_man4}, which asserts the existence of partitions of unity for smooth valuations, the restriction of this map to compactly supported elements, namely
\[q_c:\mathcal W_{k, c}(X)\to \Gamma_c(X,\Val_k^\infty(TX)),\]
is onto as well.

Note that the averaging map $A_2:\Gamma_c(X,\Val_k^\infty(TX))\to \Gamma(X,\Val_k^\infty(TX))^V$ defined analogously to $A_1$ is onto. Indeed, take any $\rho\in C^\infty_c(V)$ with $\int_V \rho(x)dx=1$. Then one has $A_2(\rho\cdot  \phi)=\phi$ for all $\phi\in \Gamma(X,\Val_k^\infty(TX))^V$. 

Now consider the commutative diagram 
\begin{equation*}
	\begin{tikzcd}
		\mathcal W_{k, c}(X) \arrow[r,"q_c"] \arrow[d,"A_1"] & \Gamma_c(X,\Val_k^\infty(TX))\arrow[d,"A_2"]   \\
		\mathcal W_{k}(X)^V\arrow[r,"q_V"] & \Gamma(X,\Val_k^\infty(TX))^V
	\end{tikzcd}
\end{equation*}
As $q_c$ and $A_2$ are both onto, the same must hold for $q_V$.

For the second assertion, we use reverse induction on $k$. The base of the induction is $k=\dim X$. As $\phi\in\mathcal W_{\dim X}(X)^V$ is a measure, we may take $\rho\in C^\infty_c(X)$ as before to get $A_1(\rho\cdot \phi)=\phi$.  Assuming the assertion for $k+1$, take  $\phi\in\mathcal W_k(X)^V$. Choose $\psi \in \mathcal W_{k, c}(X)$ with $A_2q_c\psi=q_V\phi$.
It follows that $q_VA_1\psi =q_V\phi$, that is $A_1\psi-\phi\in\mathcal W_{k+1}(X)^V$. By the induction assumption we may find $\psi'\in \mathcal W_{k+1,c}(X)$ with $A_1\psi'=A_1\psi-\phi$, so that $\phi= A_1(\psi-\psi')$.
\endproof

Let $G=K\times V$ be a product of a compact Lie group and a linear space. In invariant terms, the averaging map $A_1$ from Lemma \ref{lem:translation_surjective} with $X=G$ is $A:\mathcal V^\infty_c(G)\to \mathcal V^\infty(G)^{V}\otimes \Dens(\mathfrak g^*)$ given by  $A(\phi)=\int_V x^*\phi dx \otimes \sigma^*$, where $dx$ is an arbitrary Lebesgue measure on $V$, identified with $\sigma\in\Dens(\mathfrak g)$ via the Haar probability measure on $K$. The restriction of $A$ to $\mathcal V^\infty_c(G)^{K\times K}$  clearly coincides with the averaging map defined before Proposition \ref{mthm:compatibility}.

\proof[Proof of Theorem \ref{mthm:unimodular}] 
 We have seen in Section \ref{sec:general_convolution} that the convolution is well defined. If $\mathcal V^\infty(G)^{G\times G}=\Span\{\chi, \vol\}$, then all the properties are trivially verified. Thus we assume henceforth that $G$ admits non-obvious bi-invariant smooth valuations. If $G$ is connected, it is the product of a compact group and a linear space by Theorem \ref{mainthm_existence_biinvariant} and the convolution coincides with the one of Lemma \ref{lem:lattices}, from which associativity and continuity easily follow. The general case follows from Proposition \ref{prop:restrict_connected}. In particular, the bi-invariant valuations are closed under convolution.

The properties (i) and (ii) are immediate from Definition \ref{mdef:convolution}. As for item (iii), we have $c_\chi=1$, $\tau_\chi=0$, $\mu(\chi)=0$ and hence \eqref{eq_def_formula_convolution} gives  $(\chi\otimes \vol^*)\ast \psi=(\mu(\psi) \vol^*)\chi$. To show $\phi\ast\chi=\mu(\phi)\chi$, we use  Lemma \ref{lemma_inversion} to find that
\[\phi\ast\chi=\inv^*(\inv^*\chi\ast \inv^*\phi)=\inv^*(\chi\ast \inv^*\phi)=\mu(\inv^*\phi)\inv^*\chi=\mu(\phi)\chi. \] 

To prove (iv), write $G=K\times V$. We first note that, according to  Lemma \ref{lem:lattices}, the equality $(\phi*\psi)(e)=\mu(\phi\cdot \inv^*\psi)$ follows from the compact case, which is Lemma \ref{lem:convolution_pairing}. To verify the pairing is perfect, let $\phi\otimes\sigma\in\mathcal V^\infty(G)^{G\times G}\otimes\Dens(\mathfrak g^*)$ be non-zero. Assume that $\phi\in\mathcal W_k(G)\setminus\mathcal W_{k+1}(G)$, and let $\phi_0\in\Val^\infty_k(\mathfrak g)^{\Ad_K}$ be the corresponding principal symbol at the unit element. By translation-invariant Alesker--Poincar\'e duality and the compactness of $K$, we may find $\psi_0 \in \Val^\infty_{n-k}(\mathfrak g)^{\Ad_K}$ such that $\langle \phi_0, \psi_0\rangle \in\Dens(\mathfrak g)$ is non-zero. Let $\overline\psi \in \Gamma(G, \Val^\infty_{n-k}(TG))^{L_G}$ be the left-invariant field defined by $\psi_0$. We then use Lemma \ref{lem:translation_surjective} to find $\psi\in\mathcal V^\infty(G)^V$ with principal symbol $\overline\psi$. 
By construction, $\langle \phi\otimes \sigma, \psi\otimes \sigma\rangle\in\mathcal M(G)^G\otimes \Dens(\mathfrak g^*)^2=\Dens(\mathfrak g^*)$ is non-zero, and we may average over the action of $K\times K$ to guarantee that $\psi\in\mathcal V^\infty(G)^{G\times G}$. This concludes the proof.
\endproof

\begin{Remark} \normalfont
By construction, the convolution of Definition \ref{mdef:convolution} naturally extends the Alesker--Bernig convolution on compact Lie groups. In order to relate it to the Bernig--Fu convolution on linear spaces, we have to compare our conventions with the ones from \cite{bernig_fu06}.

First, let $*_V$ be the usual Hodge star operator on $V=\RR^n$, $\gamma \in \Omega^k(V)^{tr}$, and $\kappa \in \Omega^l(S^{n-1})$. Then the operator $*_1$ is defined in \cite{bernig_fu06} by 
	\begin{displaymath}
		*_1(\pi_1^*\gamma \wedge \pi_2^*\kappa)=(-1)^{\binom{n-k}{2}} \pi_1^*(*_V\gamma) \wedge \pi_2^* \kappa. 
	\end{displaymath}
	On the other hand, the Hodge star $*^{-1}$ from Equation \eqref{eq_def_hodge_star} satisfies 
	\begin{align*}
		*^{-1} (\pi_1^*\gamma \wedge \pi_2^*\kappa) & =(-1)^{kl} *^{-1} (\pi_2^*\kappa \wedge \pi_1^*\gamma)\\
		& =(-1)^{kl} (\pi_2^*\kappa \wedge \pi_1^*(*_V\gamma))\\
		& =(-1)^{kl} (-1)^{l(n-k)}  \pi_1^*(*_V\gamma) \wedge \pi_2^*\kappa.
	\end{align*}
	We thus have 
	\begin{equation} \label{eq_diff_hodge_stars}
		*_1=(-1)^{\binom{n-k}{2}+nl} \cdot *^{-1}
	\end{equation}
	on translation-invariant forms on $\R^n \times S^{n-1}$ of bi-degree $(k,l)$.
	
Second, the wedge product from \cite{bernig_fu06} is the one on $\R^n \times S^{n-1}$. Namely, for $\gamma_i \in \Omega^{k_i}(V)^{tr}$ and $\kappa_i \in \Omega^{l_i}(S^{n-1})$, $i=1,2$, one has
	\begin{displaymath}
		(\pi_1^*\gamma_1 \wedge \pi_2^*\kappa_1) \wedge (\pi_1^*\gamma_2 \wedge \pi_2^*\kappa_2)=(-1)^{l_1k_2} \pi_1^*(\gamma_1 \wedge \gamma_2) \wedge \pi_2^*(\kappa_1 \wedge \kappa_2).
	\end{displaymath}
	Our wedge product (temporarily denoted by $\wedge'$) satisfies 
	\begin{align*}
		&(\pi_1^*\gamma_1 \wedge \pi_2^*\kappa_1) \wedge' (\pi_1^*\gamma_2 \wedge \pi_2^*\kappa_2) \\
		&\quad =(-1)^{k_1l_1+k_2l_2} (\pi_2^*\kappa_1 \wedge \pi_1^*\gamma_1) \wedge' (\pi_2^*\kappa_2 \wedge \pi_1^*\gamma_2)\\
		& \quad=(-1)^{k_1l_1+k_2l_2} \pi_2^*(\kappa_1 \wedge \kappa_2) \wedge \pi_1^*(\gamma_1 \wedge \gamma_2)\\
		& \quad=(-1)^{k_1l_1+k_2l_2+(l_1+l_2)(k_1+k_2)}\pi_1^*(\gamma_1 \wedge \gamma_2) \wedge \pi_2^*(\kappa_1 \wedge \kappa_2).
	\end{align*}
	Hence 
	\begin{equation} \label{eq_diff_wedge_products}
		\wedge'=(-1)^{l_2k_1} \wedge
	\end{equation} 
	on translation-invariant forms on $\R^n\times S^{n-1}$ of bi-degree $(k_1,l_1), (k_2,l_2)$.
\end{Remark}

\proof[Proof of Proposition \ref{mthm:compatibility}] 

(i) Assume $G$ is compact and let $\phi \in \mathcal V(G)^{G \times G}$. Write $\phi=\{c_\phi,\tau_\phi\}=[[\mu_\phi,\omega_\phi]]$ and consider the valuation
$$\phi':=\phi \otimes \vol(G) \vol^* \in  \mathcal V(G)^{G \times G} \otimes \Dens(\mathfrak g^*).$$
Then $\phi'=\{c_\phi',\tau_\phi'\}=[[\mu_\phi',\omega_\phi']]$, where the prime forms are obtained from the original ones by multiplying with $\vol(G)\vol^*$. From \eqref{eq:Haar_identification} and \eqref{eq:Hodge-1} it follows that $\tilde \tau_\phi=\widetilde{\tau_\phi'}$ and $\tilde \omega_\phi=\widetilde{\omega_\phi'}$.
 Proposition \ref{mthm:unimodular}(i) then implies that the two convolution products coincide. 
	
(ii) Since $V$ is commutative, the higher order terms in $\tau_\phi * \omega_\psi$ and $\tau_\phi * \tau_\psi$ vanish and \eqref{eq_convolution_formula} reduces to \cite[Equations (37)--(39)]{bernig_fu06}, taking into account \eqref{eq_diff_hodge_stars} and \eqref{eq_diff_wedge_products}.
	
(iii)  By \cite[Section 3.1]{alesker_val_man2} there is an isomorphism of graded vector spaces
	\begin{equation} \label{eq_graded_isomorphism_vals}
		I=\bigoplus_{k=0}^n I_k:\mathrm{gr} \ (\mathcal V^\infty(G)^{G \times G}) \cong \Val^\infty(\mathfrak g)^{\Ad G}
	\end{equation} 
	that can be described in terms of differential forms as follows.  Let $[\phi] \in \mathcal W_k(G)^{G \times G}/\mathcal W_{k+1}(G)^{G \times G}$ denote the equivalence class of $\phi \in  \mathcal W_k(G)^{G \times G}$. 
	
	If $\phi \in \mathcal W_0(G)^{G \times G}$, then $I_0([\phi])=\phi(\{e\}) \chi$. Let $\phi=\{0,\tau\} \in \mathcal W_k(G)^{G \times G}$ with $1 \leq k \leq n$ and $\tau \in \Omega^n_{\ori}(G \times \p_+(\mathfrak g^*))^{G \times G}$. Let  $\tau=\sum_{i=0}^{n-k} \tau_i$ with $\tau_i \in \Omega^{n-i,i}_{\ori}(G \times \p_+(\mathfrak g^*))$ be the decomposition of $\tau$. The restriction of $\tau_{n-k}$ to the sphere over the identity element of $G$ can be identified with a closed, vertical, translation-invariant, and $\Ad G$-invariant form in $\Omega^{k,n-k}_{\ori}(\mathfrak g \times \p_+(\mathfrak g^*))$, which defines $I_k([\phi]) \in \Val_k^\infty(\mathfrak g)^{\Ad G}$.
	
	Similarly, if $\psi=[[0,\omega]]  \in \mathcal W_l(G)^{G \times G}$ with $1 \leq l \leq n-1$ and $\omega \in \Omega^{n-1}_{\ori}(G \times \p_+(\mathfrak g^*))^{G \times G}$, then we decompose $\omega=\sum_{i=0}^{n-1-l} \omega_i$ with $\omega_i \in \Omega_{\ori}^{n-1-i,i}(G \times \p_+(\mathfrak g^*))$. 
	
	The restriction of $\omega_{n-l-1}$ to the sphere over the identity defines a translation-invariant, $\Ad G$-invariant form in $\Omega_{\ori}^{l,n-1-l}(\mathfrak g \times \p_+(\mathfrak g^*))$, which is the form defining $I_l([\psi])$. 
	
	Tensorizing with $\Dens(\mathfrak g^*)$ we get an isomorphism of graded vector spaces
	\begin{equation} \label{eq_graded_isomorphism_vals_compact}
		\tilde I=\bigoplus_{k=0}^n \tilde I_k:\mathrm{gr} \ (\mathcal V(G)^{G \times G} \otimes \Dens(\mathfrak g^*)) \cong \Val^\infty(\mathfrak g)^{\Ad G} \otimes \Dens(\mathfrak g^*).
	\end{equation} 
	
	We claim that \eqref{eq_graded_isomorphism_vals_compact} is compatible with the convolution products on both sides, i.e., that for  any $\phi \in \mathcal W_k(G)^{G \times G} \otimes \Dens(\mathfrak g^*)$ and $\psi \in \mathcal W_l(G)^{G \times G} \otimes \Dens(\mathfrak g^*)$ we have 
	\begin{equation} \label{eq_compatibility_convolution_products}
		\tilde I_{k+l-n}([\phi * \psi])=\tilde I_k([\phi]) * \tilde I_l([\psi]).
	\end{equation}
	Note that $\phi * \psi \in \mathcal W_{k+l-n}(G)^{G \times G} \otimes \Dens(\mathfrak g^*)$, so the left-hand side is well defined.  
	
	If $k+l<n$, then both sides of the equation vanish. If $k=n$ or $l=n$, then \eqref{eq_compatibility_convolution_products} is immediate from the fact that $\vol \otimes \vol^*$ is the unit in $\mathcal V^\infty(G)^{G \times G} \otimes \Dens(\mathfrak g^*)$, and its image under $\tilde I_n$ is the unit in $\Val^\infty(\mathfrak g^*) \otimes \Dens(\mathfrak g^*)$. 
	
	Suppose next that $k+l \geq n+1$. Note first that the lowest degree component of $\tau_\phi * \tau_\psi$ equals $(\widetilde{\tau_\phi})_{n-k} \wedge (\widetilde{\tau_\psi})_{n-l}$, since $\epsilon^{n,n}_{n-k,n-l,2n-k-l}=1$. Up to signs, the claim is then obvious from Proposition \ref{prop_convolution_formula}, \cite[Proposition 2.7]{bernig_fu06}, and the above description of the map $\tilde I$ in terms of differential forms. Using \eqref{eq_diff_hodge_stars} and \eqref{eq_diff_wedge_products} one easily checks that the signs match.
	
	Finally, suppose that $k+l=n$. The lowest degree component of $\tau_\phi * \omega_\psi$ equals $(-1)^{n+k} (\widetilde{\tau_\phi})_{n-k} \wedge (\widetilde{\omega_\psi})_{n-1-l}$, since $\epsilon^{n,n-1}_{n-k,n-l-1,n-1}=(-1)^{n+k}$. The claim then follows from \cite[Equation (37)]{bernig_fu06}. 

(iv) First note that since $V$ lies in the centralizer of $\mathfrak g$, we have $A(\phi\ast\psi)=A\phi\ast \psi=\phi\ast A\psi$. We will show that $A(\phi\ast\psi)=A\phi\ast A\psi$. 

	Choose a lattice $\Gamma \subset V$ and denote $T=V/\Gamma$, $G_T=K\times T$. Let $Q:\mathcal V^\infty(G_\Gamma)^{G_\Gamma\times G_\Gamma}\to \mathcal V^\infty(G)^{G\times G}\otimes \Dens(\mathfrak g^*)$ be as in Definition \ref{mdef:convolution}. 
	Write $A=Q\circ A_T\circ A_\Gamma$, where $A_\Gamma:\mathcal V^\infty_c(G)^{K\times K}\to \mathcal V^\infty(G_\Gamma)^{K\times K}$ is $A_\Gamma\phi=\sum_{\gamma\in\Gamma}\gamma^*\phi$, and $A_T: \mathcal V^\infty(G_\Gamma)^{K\times K}\to \mathcal V^\infty(G_\Gamma)^{G_\Gamma\times G_\Gamma}$ is averaging over the torus. 
	
	Write $\theta=A\phi=Q\theta_T$, so that $\theta_T=A_TA_\Gamma\phi$. We have
	\[\theta\ast A\psi=Q\theta_T\ast QA_TA_\Gamma\psi=Q(\theta_T\ast A_TA_\Gamma\psi)=QA_T(A_\Gamma\phi\ast A_\Gamma\psi).\]
	
	It remains to check that $A_\Gamma\phi\ast A_\Gamma\psi=A_\Gamma(\phi\ast\psi)$.  Indeed $A_\Gamma\phi\ast \psi=A_\Gamma(\phi\ast \psi)$ as $\Gamma$ lies in the centralizer of $\mathfrak g$. The remaining equality $A_\Gamma\phi\ast \psi=A_\Gamma\phi\ast A_\Gamma\psi$ is transparent if $\spt \psi$ lies in a fundamental domain of $\Gamma$, and we may reduce to this case by linearity and the partition of unity property of valuations. 
	
Finally, by Lemma \ref{lem:translation_surjective} we have $\eta=A\psi$ for some compactly supported $\psi$, whence \[A\phi\ast\eta=A\phi\ast A\psi=\phi\ast A\psi=\phi\ast \eta.\] 
\endproof

\section{Convolution on $S^3$}

\label{s:S3}

By Proposition \ref{prop_finite_dimensionality}, there are only two compact Lie groups with a finite-dimensional algebra of bi-invariant valuations, namely $S^3$ and $\mathrm{SO}(3)$. As $\mathrm{SO}(3)$ is a quotient of $S^3$, we will only describe the convolution algebra of bi-invariant valuations on $S^3$ and leave the easy modifications in the other case to the reader.

It is well known that the space $\calV(S^3)^{S^3\times S^3}$ is $4$-dimensional and spanned by the intrinsic volumes $\mu_i$, $i=0,\dots,3$, see \cite{bernig_faifman_solanes_part2}. 

\begin{Lemma}
	\label{lem:commutativity}
	The convolution on $\mathcal V^\infty(S^3)^{S^3 \times S^3}$ is commutative. 
\end{Lemma}

\begin{proof}
	Since the group inverse $\inv$ is an isometry, we have $\inv^*\mu_i=\mu_i$ and the claim follows from Lemma \ref{lemma_inversion}
\end{proof}

Let us denote $\chi_X(Y)=\chi(X\cap Y)$, for $X,Y\in\calP(S^3)$, and introduce the Crofton valuations. 
\begin{align*}
	\nu_0&=\chi,\\
	\nu_1&=\int_{\Gr_2(S^3)}\chi_EdE,\\
	\nu_2&=\int_{\Gr_1(S^3)}\chi_LdL,\\
	\nu_3&=\frac{1}{2\pi^2}\mu_3.
\end{align*}
Here $\Gr_j(S^3)$ is the space of $j$-dimensional great spheres, and $dE, dL$ are the Haar probability measures.

\begin{Lemma} \label{lemma_dictionnary_mu_nu}
	One has
\begin{align}
\label{eq:mu0}	\mu_0 & =\nu_0,\\
\label{eq:mu1}	\mu_1 & =\pi \nu_1+ \pi \nu_3,\\
\label{eq:mu2}	\mu_2 & =2\pi \nu_2,\\
\label{eq:mu3}	\mu_3 & = 2\pi^2 \nu_3.
\end{align}
In particular, the Crofton valuations $\nu_0,\dots,\nu_3$ form a basis of $\calV^\infty(S^3)^{S^3 \times S^3}$.
\end{Lemma}

\proof
First, \eqref{eq:mu0} and \eqref{eq:mu3} are obvious. Second, since $\mu_0,\mu_2,\nu_2$ are eigenvectors of the Euler--Verdier involution with eigenvalue 1 and $\mu_1,\mu_3,\nu_1$ eigenvectors with eigenvalue -1, there are $a,b,c,d\in\CC$ such that
\begin{align*}
\nu_1&=a\mu_1+b\mu_3,\\
\nu_2&=b\mu_0+d\mu_2.
\end{align*}
We will use the template method to determine the constants. To this end, recall that for any $0\leq i\leq n$ one has
\begin{align*}
\chi(S^i)&=1+(-1)^i,\\
\mu_i(S^i)&=(i+1)\omega_{i+1},\\
\mu_i(B^n)&=\binom ni\frac{\omega_n}{\omega_{n-i}},
\end{align*}
see \cite[Theorem 9.2.4]{klain_rota}. It follows that 
$$
\nu_1(S^1)=\chi(S^0)=2, \quad \mu_1(S^1)=2\pi,\quad \mu_3(S^1)=0,
$$
and
$$
\nu_1(S^3)=\chi(S^2)=2, \quad \mu_1(S^3)=2\mu_1(B^4)=3\pi,\quad \mu_3(S^3)=2\pi^2.
$$
Consequently, $2=2\pi a$ and $2=3\pi a+2\pi^2 b$, and \eqref{eq:mu1} follows. Similarly, for any point $p\in S^3$,
$$
\nu_2(\{p\})=0, \quad \mu_0(\{p\})=1,\quad \mu_2(\{p\})=0,
$$
and
$$
\nu_2(S^2)=\chi(S^0)=2, \quad \mu_0(S^2)=2,\quad \mu_2(S^2)=4\pi
$$
gives $c=0$ and $2=2c+4\pi d$, and proves \eqref{eq:mu2}.
\endproof

\begin{Lemma} \label{lemma_ev}
	Let $G$ be a compact Lie group. If $\phi,\psi \in \mathcal V^\infty(G)^{G\times G}$ are eigenvectors of the Euler--Verdier involution with eigenvalues $\epsilon_\phi,\epsilon_\psi$, respectively, then $\phi* \psi$ is an eigenvector with eigenvalue $(-1)^{\dim G}\epsilon_\phi\epsilon_\psi$.
\end{Lemma}
\proof
Clearly, $\phi\boxtimes \psi$ has eigenvalue $\epsilon_\phi\epsilon_\psi$. Under the push-forward by the multiplication map $m:G\times G\to G$, the Euler--Verdier eigenvalue shifts by $(-1)^{\dim (G\times G)-\dim G} =(-1)^{\dim G}$.
\endproof

\begin{Proposition}\label{prop:nu_table}
	In $\mathcal V^\infty(S^3)^{S^3 \times S^3}$, the following identities hold:
	\begin{align}
\label{eq:nuk3}		\nu_k * \nu_3 & = \nu_k, \quad k=0,\ldots,3,\\
\label{eq:nu11}			\nu_1\ast\nu_1&=4\nu_3,\\
\label{eq:nu00}			\nu_0\ast\nu_0&=0,\\
\label{eq:nu02}			\nu_0\ast\nu_2&=0,\\
\label{eq:nu22}			\nu_2\ast \nu_2&=\frac{\pi^2}{4}(\nu_1-2\nu_3),\\
\label{eq:nu01}			\nu_0\ast\nu_1&=2\nu_0,\\
\label{eq:nu12}			\nu_1\ast\nu_2&=2\nu_0-2\nu_2.
	\end{align}
\end{Proposition}

\proof
Since $\nu_3$ is the Haar probability measure on $S^3$, \eqref{eq:nuk3} follows at once from Theorem \ref{mthm:unimodular}(i) and  Proposition \ref{mthm:compatibility}(i). To prove the remaining relations, we will keep the notation from the proof of Lemma \ref{lem:commutativity} and exploit a version of the template method used in the proof of Lemma \ref{lemma_dictionnary_mu_nu}. To this end, Lemma \ref{lemma_ev} implies that
\begin{align*}
	\nu_1\ast\nu_1&=a_1\nu_1+a_3\nu_3,\\
	\nu_0\ast\nu_0&=b_1\nu_1+b_3\nu_3,\\
	\nu_0\ast\nu_2&=c_1\nu_1+c_3\nu_3,\\
	\nu_2\ast \nu_2&=d_1\nu_1+d_3\nu_3,\\
	\nu_0\ast\nu_1&=a_0\nu_0+a_2\nu_2,\\
	\nu_1\ast\nu_2&=b_0\nu_0+b_2\nu_2
\end{align*}
for some constants $a_i,b_i,c_i,d_i\in\CC$. By definition, for all $A,B,X\in\calP(S^3)$ we have $\chi_A\ast \chi_B(X)= \chi_{A\times B}(m^{-1}X)$.

Take $X=S^3$. Then $\chi_A\ast \chi_B(S^3)=\chi_{A\times B}(S^3\times S^3)=\chi(A)\chi(B)$, so 
$$ \nu_0\ast\nu_0(S^3)=0,\quad \nu_0\ast\nu_2(S^3)=0,\quad \nu_1\ast\nu_1(S^3)=4,\quad\nu_2\ast \nu_2(S^3)=0.$$
As $\nu_1(S^3)=2$ and $\nu_3(S^3)=1$, we get the equations
\begin{align*}
	2a_1+a_3&=4,\\
	2b_1+b_3&=0,\\
	2c_1+c_3&=0,\\
	2d_1+d_3&=0.
\end{align*}

Now consider $X=S^1$ and note that $\nu_1(S^1)=2$, $\nu_j(S^1)=0$ for $j\neq 1$. First, to evaluate $\nu_1\ast\nu_1(S^1)$, take $E, E'\in\Gr_2(S^3)$ and denote $E=u^\perp$, $E'=\inv(v^\perp)$. For any $\theta\in X$, the sets $E$ and $\theta \cdot \mathrm{Inv}(E')$ intersect along a circle unless $u=\pm\theta v$. So if $u\not\in X v$, which is the generic case, the set $E\times E'\cap m^{-1}X$ fibers over $X$ with fiber $\{(x,x^{-1}\theta) \mid x\in \theta\cdot\mathrm{Inv}(E')\cap E\}=S^1$ over $\theta\in X$. Hence $\chi(E\times E'\cap m^{-1}X)=0$ for generic pairs $(E, E')$, and so
$$\nu_1\ast\nu_1(S^1)=0.$$
Thus we get the equation $2a_1=0$ which implies $a_1=0$ and $a_3=4$, proving \eqref{eq:nu11}.  Second, since $m^{-1}S^1=\{(x,x^{-1}\theta)\mid x \in S^3,\theta\in S^1\}\xrightarrow{S^1}S^3$, we have
\begin{displaymath}
	\nu_0\ast\nu_0(S^1)=\chi(m^{-1}S^1)=\chi(S^1)\chi(S^3)=0,
\end{displaymath}
so that $2b_1=0$ and so $b_1=b_3=0$, proving \eqref{eq:nu00}.  Similarly, we note that $S^3\times L\cap m^{-1}X$ fibers over $L \in \Gr_1(S^3)$ with fiber equal to $S^1$. Consequently, $\chi(S^3\times L\cap m^{-1}X)=0$ and so
$$\nu_0*\nu_2(S^1)=0.$$
It follows that $ 2c_1=0$ and hence $c_1=c_3=0$, proving \eqref{eq:nu02}.

To finish the proof of \eqref{eq:nu22}, we will use a different argument. Note that by Lemma \ref{lemma_dictionnary_mu_nu},
\begin{displaymath}
	\mu_2\ast\mu_2 =d_1(4\pi \mu_1-6\mu_3).
\end{displaymath}
Considering the standard volume on $S^3$ with total mass $2\pi^2$ and the canonical Lebesgue measure on $\RR^3$, the isomorphism from Proposition \ref{mthm:compatibility}(iii) takes the class $[\mu_k] \in \mathcal W^\infty_k(S^3)^{S^3 \times S^3}/\mathcal W^\infty_{k+1}(S^3)^{S^3 \times S^3}$ to $2\pi^2 \mu_k \in \Val_k(\R^3)$ and implies that on $\R^3$,
\begin{displaymath}
	2\pi^2 \mu_2 \ast 2\pi^2 \mu_2 =d_1 4\pi \cdot 2\pi^2 \mu_1.
\end{displaymath}
On the other hand, from \cite[Corollary 1.3]{bernig_fu06} we know that on $\R^3$,
\begin{displaymath} 
	\mu_2\ast \mu_2=\frac{\pi}{2}\mu_1.
\end{displaymath}
Hence $d_1=\frac{\pi^2}{4}$ and \eqref{eq:nu22} follows.

Next, fix $p\in S^3$ and take $X=\{p\}$. Since the inverse map is an isometry, we find that $\nu_0\ast\nu_1(\{p\})=2 $ and $\nu_1\ast \nu_2(\{p\})=2$. Thus  $a_0=b_0=2$.

Finally take $X=S^2$. For any $E\in\Gr_2(S^3)$ we have
	\begin{align*}
	\chi(S^3\times E\cap m^{-1}X)&=\chi\{(zy^{-1},y) \mid z\in X, y\in E\}\\
	&=\chi(\{(z,y) \mid z\in X, y\in E\})\\
	&=\chi(E)\chi(X)\\
	&=4.
	\end{align*}
	Consequently, $\nu_0\ast\nu_1(S^2)=4$, and so $2a_0+2a_2=4$ implies $a_2=0$ and \eqref{eq:nu01} follows. We will then reduce the computation of $\nu_1\ast\nu_2(S^2)$ to the case $\nu_1\ast\nu_1(S^1)$ as follows. Fix $E\in\Gr_2(S^3)$ and $L\in\Gr_1(S^3)$. Then
	\begin{align*}
	\chi(E\times L\cap m^{-1}X)&=\chi\{(x,y) \mid x\in E, y\in L, xy\in X\}\\
	&=\chi\{(x^{-1}, xy) \mid x\in E, y\in L, xy\in X\}\\
	&=\chi\{(x^{-1},z) \mid x^{-1}\in \inv(E), z\in X, x^{-1}z\in L\}\\
	&=\chi(\inv(E)\times X\cap m^{-1}L).
	\end{align*}
	Since the inverse map is an isometry, $E'=\inv(E)$ is another geodesic $2$-sphere. Consider the isometry group $\mathrm{Isom}(S^3)=\mathrm{SO}(4)$. Now one can write
	\begin{align*}
	\nu_1\ast\nu_2(X)&=\int_{\mathrm{Isom}(S^3)\times \mathrm{Isom}(S^3)\times \mathrm{Isom}(S^3)}\chi(g_1E\times g_2L\cap m^{-1}g_3X)\\
	&=\int_{\mathrm{Isom}(S^3)\times \mathrm{Isom}(S^3)\times \mathrm{Isom}(S^3)}\chi(g_1E\times g_2X\cap m^{-1}g_3L)\\
	&=\nu_1\ast\nu_1(L)\\
	&=0.
	\end{align*}
	Thus $2b_0+2b_2=0$ which implies $b_2=-2$ and \eqref{eq:nu12} follows.
\endproof

Combining Proposition \ref{prop:nu_table} and Lemma \ref{lemma_dictionnary_mu_nu} yields the convolution table in terms of intrinsic volumes as follows.
\begin{Corollary}
	\begin{align*}
		\mu_k \ast \mu_3 & = 2\pi^2 \mu_k,\quad k=0,\ldots,3,\\
		\mu_1\ast\mu_1&= 2\pi\mu_1+\frac32\mu_3,\\
		\mu_0\ast\mu_0&=0,\\
		\mu_0\ast\mu_2&=0,\\
		\mu_2\ast\mu_2&=\frac{\pi^2}{4}(4\pi \mu_1-6\mu_3),\\
		\mu_0\ast\mu_1&=3\pi \mu_0,\\
		\mu_1\ast\mu_2&=4\pi^2\mu_0-\pi\mu_2.
	\end{align*}
\end{Corollary}

On a flat space, the Alesker--Fourier transform is an isomorphism between the algebra of smooth valuations with respect to the Alesker product and the algebra of smooth valuations with respect to the convolution product. An analogous isomorphism for $S^3$ does not exist. 
	
	\begin{Corollary}
		The map $t \mapsto \pi^{-1} \nu_2$ covers an isomorphism of algebras 
		\begin{displaymath}
			\C[t]/(t^2+t^4) \to  \left(\mathcal V^\infty(S^3)^{S^3 \times S^3},*\right).
		\end{displaymath}
		In particular, there is no algebra isomorphism between the convolution algebra $\left(\mathcal V^\infty(S^3)^{S^3 \times S^3},*\right)$ and the product algebra $\left(\mathcal V^\infty(S^3)^{S^3 \times S^3},\cdot\right)$.
	\end{Corollary}

\proof
The first statement follows from Proposition \ref{prop:nu_table}. For the second statement, recall that the product algebra is isomorphic to $\C[t]/(t^4)$ \cite{fu_wannerer}, hence both algebras are of dimension $4$, with the equivalence classes of $1,t,t^2,t^3$ being a basis.
Let us compare their nilradicals. In the product algebra, the nilradical is the ideal $(t)$, which is of dimension $3$. In the convolution algebra, the nilradical is contained in the $3$-dimensional ideal $(t)$. However, as $(t^2)^k=(-1)^{k+1} t^2 \neq 0$ for all $k \geq 1$, it is of dimension strictly less than $3$.  
\endproof

\bibliographystyle{plain}
\bibliography{biblio}
\end{document}